\documentclass{amsart}
\usepackage[centertags]{amsmath}
\usepackage{amsfonts}
\usepackage{amssymb}
\usepackage{amsthm}
\usepackage{newlfont}
\usepackage{amscd}
\usepackage{mathrsfs}
\usepackage[all,cmtip]{xy}
\usepackage{tikz}
\usepackage[pagebackref=true,colorlinks]{hyperref} 
\hypersetup{
    colorlinks=true,
pdffitwindow=true,
linkcolor=blue,
citecolor=blue,
urlcolor=blue
}
\usepackage{cite}
\usepackage{fancyhdr}
\usepackage{stmaryrd}

\usepackage[OT2,OT1]{fontenc}
\newcommand\cyr{%
\renewcommand\rmdefault{wncyr}%
\renewcommand\sfdefault{wncyss}%
\renewcommand\encodingdefault{OT2}%
\normalfont
\selectfont}
\DeclareTextFontCommand{\textcyr}{\cyr}

\newtheorem{thm}{Theorem}[section]
\newtheorem{cor}[thm]{Corollary}
\newtheorem{lem}[thm]{Lemma}
\newtheorem{prop}[thm]{Proposition}

\newtheorem{assu}[thm]{Assumption}

\theoremstyle{definition}
\newtheorem{defn}[thm]{Definition}
\newtheorem{rem}[thm]{Remark}

 \newcommand{\Q}{\mathbb{Q}}

\newcommand{\C}{\mathbb{C}}

 \numberwithin{equation}{section}

\begin{document}
\title{On the quantitative variation of congruence ideals and integral periods of modular forms}
\author{Chan-Ho Kim}
\address{(C.-H. Kim) Center for Mathematical Challenges, Korea Institute for Advanced Study, 85 Hoegiro, Dongdaemun-gu, Seoul 02455, Republic of Korea}
\email{chanho.math@gmail.com}
\author{Kazuto Ota}
\address{(K. Ota) Department of Mathematics, Osaka University, 1-1 Machikaneyama, Toyonaka, Osaka 560-0043, Japan}
\email{kazutoota@math.sci.osaka-u.ac.jp}
\date{\today}
\thanks{Data Availability Statements: Data sharing not applicable to this article as no datasets were generated or analysed during the current study.}
\thanks{Statements and Declarations: The authors have no relevant financial or non-financial interests to disclose.}

\begin{abstract}
We prove the conjecture of Pollack and Weston on the quantitative analysis of the level lowering congruence \`{a} la Ribet for modular forms of higher weight.
It was formulated and studied in the context of the integral Jacquet--Langlands correspondence and anticyclotomic Iwasawa theory for modular forms of weight two and square-free level for the first time.
We use a completely different method based on the $R=\mathbb{T}$ theorem established by Diamond--Flach--Guo and Dimitrov and an explicit comparison of adjoint $L$-values.
We briefly discuss arithmetic applications of our main result at the end.
\end{abstract}

\maketitle

\setcounter{tocdepth}{1}

\section{Introduction} \label{sec:introduction}
\subsection{Overview} \label{subsec:overview}
\subsubsection{Main result}
The goal of this article is to prove Theorem \ref{thm:main_thm} below, which is the higher weight generalization of the conjecture of Pollack and Weston  \cite[Conjecture 1.4]{pw-mu}.

Let $p \geq 5$ be a prime and fix embeddings
$\iota_p : \overline{\mathbb{Q}} \hookrightarrow \overline{\mathbb{Q}}_p$ and
$\iota_\infty : \overline{\mathbb{Q}} \hookrightarrow \mathbb{C}$.
Let $f = \sum_{n \geq 1} a_n(f) q^n \in S_{k}(\Gamma_0( N ))$ be a newform and $\mathbb{Q}_f$ the Hecke field of $f$.
Let $\lambda$ be the prime of $\mathbb{Q}_f$ lying above $p$ induced from $\iota_p$.
Let $E = \mathbb{Q}_{f, \lambda}$ be the completion at $\lambda$, $\mathcal{O} = \mathcal{O}_E$, and $\mathbb{F} = \mathcal{O} /\lambda\mathcal{O}$ the residue field.
Following Deligne's construction, there exists a continuous $\lambda$-adic Galois representation arising from $f$ (and $\iota_p$)
$$\rho_f : G_{\mathbb{Q}} = \mathrm{Gal}(\overline{\mathbb{Q}}/\mathbb{Q}) \to \mathrm{Aut}_{E}(V_{f}) \simeq \mathrm{GL}_2(E)$$
which is unramified outside $Np$ and is characterized by $\mathrm{tr}(\rho_{f}(\mathrm{Frob}_\ell)) = \iota_p (a_\ell(f))$
for each prime $\ell$ not dividing $Np$ where $\mathrm{Frob}_\ell$ is a geometric Frobenius at $\ell$.

By choosing a Galois stable $\mathcal{O}$-lattice of $V_f$, the mod $\lambda^n$ representation $\rho_n := \rho_{f} \pmod{\lambda^n}$ and the residual representation $\overline{\rho} = \overline{\rho}_{f} := \rho_1$ are defined, and these are unique up to scalars if $\overline{\rho}$ is irreducible.
The Tamagawa exponent $t_f(q)$ of $f$ at a prime $q$ dividing $N$ is defined by the largest integer such that $\rho_{t_f(q)}$ is unramified at $q$ and $\rho_{t_f(q)+1}$ is ramified at $q$. See $\S$\ref{subsec:tamagawa_exponents} for more details.
\begin{thm}[Main Theorem] \label{thm:main_thm}
Decompose the level $N$ of $f$ by $N = N^+N^-$ and assume the following conditions:
\begin{enumerate}
\item $(N,p)=1$,
\item the restriction $\overline{\rho}_f \vert_{ \mathrm{Gal}(\overline{\mathbb{Q}}/ \mathbb{Q}(\sqrt{p^*}))  }$ is absolutely irreducible, where $p^* = (-1)^{\frac{p-1}{2}}p$,
\item $2 \leq k \leq p-1$,
\item  $(N^+, N^-)=1$,
\item $N^-$ is square-free, and
\item if a prime $q \equiv \pm 1 \pmod{p}$ and $q$ divides $N^-$, then $\overline{\rho}_f$ is ramified at $q$.
\end{enumerate}
Then we have
\begin{equation} \label{eqn:quantitative_level_lowering_formula_higher}
\mathrm{ord}_\lambda \eta_{f}(N) = \mathrm{ord}_\lambda \eta_{f}(N^+, N^-) +  \sum_{q \vert N^-} t_f(q) 
\end{equation}
where $\eta_f(N)$ is the congruence ideal of $f$ in $S_k(\Gamma_0(N))$
and $\eta_f(N^+, N^-)$ is the congruence ideal of $f$ in the $N^-$-new subspace $S_k(\Gamma_0(N))^{N^-}$ of $S_k(\Gamma_0(N))$
 (reviewed in $\S$\ref{subsec:modularity_lifting_congruence_ideals}). 
\end{thm}
\begin{rem}
\begin{enumerate}
\item 
 The formula (\ref{eqn:quantitative_level_lowering_formula_higher}) quantifies the level lowering congruences \cite{inv100} while Wiles' numerical criterion \cite[Theorem 2.17]{wiles} quantifies the level raising congruences (cf.~\cite[Theorem 5.3]{ddt}).
\item When $k=2$ and $N$ is square-free, (\ref{eqn:quantitative_level_lowering_formula_higher}) is first formulated in the context of anticyclotomic Iwasawa theory for modular forms, and is proved if $\overline{\rho}_f$ is ramified at at least two primes \cite{pw-mu}. 
The approach of Pollack and Weston is based on the work of Ribet--Takahashi \cite{ribet-takahashi, takahashi-jnt} on the comparison among the parametrizations by modular and Shimura curves, and it is unclear how to generalize this geometric approach when $k >2$ or $N$ is not square-free.
In the higher weight case, by using an $R=\mathbb{T}$ theorem for quaternion algebras, Chida--Hsieh \cite{chida-hsieh-p-adic-L-functions} proved (\ref{eqn:quantitative_level_lowering_formula_higher}), assuming  
$\overline{\rho}_f$ is ramified at every prime $q\mid N^-$, i.e. $\sum_{q \vert N^-} t_f(q) = 0$.
Our contribution is to weaken all these assumptions of the previous works on weight, level, and the ramification. 
\end{enumerate}
\end{rem}

\subsubsection{Our approach}
Since the comparison among various congruence ideals is closely related to the freeness of the (quaternionic) Hecke modules over the associated Hecke algebras, it seems inevitable to use the $R=\mathbb{T}$ technique to obtain the equality (\ref{eqn:quantitative_level_lowering_formula_higher}).
However, when $\sum_{q \vert N^-} t_f(q) \neq 0$, the standard $R=\mathbb{T}$ approach is not strong enough since it is difficult to write down $\eta_{f}(N^+, N^-)$ in terms of the size of a certain adjoint Selmer group. 
In particular, it is unclear how to figure out the \emph{correct} local (deformation) condition of the adjoint Selmer group at a prime $q$ dividing $N^-$ under our setting. 

Our approach consists of two steps. 
We first use standard $R=\mathbb{T}$ arguments and Galois cohomology calculations to relate a certain congruence ideal and the size of an adjoint Selmer group with  \emph{relaxed} local condition at primes dividing $N^-$.
We also compute the size of an adjoint Selmer group with ``new" local (deformation) condition at primes dividing $N^-$ by establishing a slightly refined $R=\mathbb{T}$ theorem.
Unfortunately, the comparison of the sizes of these two adjoint Selmer groups does not give us the equality (\ref{eqn:quantitative_level_lowering_formula_higher}) exactly.
We call this comparison the \emph{Selmer computation}.

In order to remove the difference between  (\ref{eqn:quantitative_level_lowering_formula_higher}) and the Selmer computation, we interpret the difference in terms of Euler factors of the adjoint $L$-values, which is the analytic aspect of congruence ideals.
We call this process the \emph{$L$-value computation}.
It is unexpected that we do not fix any quaternion algebra in the argument and we only use classical $R=\mathbb{T}$ theorems established in \cite{diamond-flach-guo, dimitrov-ihara}.

\subsubsection{Applications}
In $\S$\ref{sec:appendix}, we briefly discuss arithmetic applications of Theorem \ref{thm:main_thm}.
It includes the comparison between Hida's canonical periods and Gross periods (e.g. \cite{haining-heegner-cycles}), the $\mu$-part of the anticyclotomic main conjecture for modular forms, the relation between the periods of modular and (indefinite!) Shimura curves, and more.

\subsubsection{Related results}
After the completion of the first version of this paper, Fakhruddin--Khare--Ramakrishna released a preprint (since published as \cite{fakhruddin-khare-ramakrishna-quantitative}) in which they study the quantitative aspect of level lowering congruences in a \emph{different} context. Their work mainly  is concerned with producing optimal mod $p^n$ congruent modular forms (of weight two and square-free level), but our work is concerned with describing the difference between two congruence ideals in terms of purely local numerical invariants, Tamagawa exponents.

\subsection{Organization}
We describe the background material to give a precise description of Theorem \ref{thm:main_thm} and explain the idea of the proof of Theorem \ref{thm:main_thm} in $\S$\ref{sec:congruence_ideals_tamagawa_exponents}.
We review deformation theory of Galois representations in $\S$\ref{sec:deformation_theory}. Especially, all the local deformation conditions we use are precisely described and a presentation of the deformation ring is discussed.
We prove a refined $R=\mathbb{T}$ theorem (Theorem \ref{thm:refined_R=T}) at the end.
In $\S$\ref{sec:adjoint_Selmer}, we compute the difference of two adjoint Selmer groups with different local conditions in terms of purely local invariants by using Theorem \ref{thm:refined_R=T}. This completes the Selmer computation (Theorem \ref{thm:key_thm}).
In $\S$\ref{sec:adjoint_L-values}, we study congruence ideals in terms of adjoint $L$-values and the $L$-value computation (Proposition \ref{prop:second_correction}) is proved. Thus, we obtain a proof of Theorem \ref{thm:main_thm}.
Arithmetic applications are discussed in $\S$\ref{sec:appendix}.

\subsection*{Acknowledgement}
We would like to thank
Ming-Lun Hsieh,
Shinichi Kobayashi,
Robert Pollack,
Ken Ribet,
Ryotaro Sakamoto, and
Sho Yoshikawa for helpful discussions and suggestions.
Chan-Ho Kim is partially supported 
by a KIAS Individual Grant (SP054102) via the Center for Mathematical Challenges at Korea Institute for Advanced Study and
by the National Research Foundation of Korea(NRF) grant funded by the Korea government(MSIT) (No. 2018R1C1B6007009).
Kazuto Ota is partially supported by JSPS KAKENHI Grant Numbers JP17K14173 and JP18J01237.
We deeply thank the referees for their careful reading and extremely helpful comments.
The exposition is greatly improved and many inaccuracies are removed thanks to their comments.

\section{Congruence ideals and Tamagawa exponents} \label{sec:congruence_ideals_tamagawa_exponents}
In this section, we first introduce some notation on congruence ideals and review the $R=\mathbb{T}$ theorem.
After that we precisely state two key results (Theorems \ref{thm:key_thm} and Proposition \ref{prop:second_correction}), whose proofs are given in the following sections. 
\subsection{Modularity lifting theorem and congruence ideals} \label{subsec:modularity_lifting_congruence_ideals}
We review deformation rings and Hecke algebras and introduce two auxiliary level structures.
Although these level structures are complicated at first, each level structure is needed for the Selmer computation and the $L$-value computation, respectively. 
Then we recall the notion of congruence ideals and the $R=\mathbb{T}$ theorem, and define the $N^-$-new variant of congruence ideals.
\subsubsection{Deformation rings and Hecke algebras}
Let 
$$\overline{\rho} : G_{\mathbb{Q}} \to \mathrm{GL}_2(\mathbb{F})$$
 be an odd, absolutely irreducible, and continuous Galois representation where $\mathbb{F}$ is a finite extension of $\mathbb{F}_p$. Denote by $N(\overline{\rho})$ the conductor of $\overline{\rho}$. 
Then $\overline{\rho}$ is modular by \cite{khare-wintenberger-1} and \cite{kisin-2-adic-barsotti-tate}.

Following \cite{diamond-flach-guo, dimitrov-ihara}, we assume the following conditions throughout this article.
\begin{assu} \label{assu:working_assumptions}
\begin{enumerate}
\item[(FL)]
There exists a newform $f = \sum_{n \geq 1} a_n(f)q^n \in S_k(\Gamma_0(N) )$ such that $2 \leq k \leq p-1$, $p \nmid N$, and $\overline{\rho}_f \simeq \overline{\rho} $ where $\overline{\rho}_f$ is the residual Galois representation associated to $f$.
\item[(TW)]
The restriction $\overline{\rho} \vert_{G_{\mathbb{Q}(\sqrt{p^*})}}$ is absolutely irreducible where $p^* = (-1)^{\frac{p-1}{2}}p$.
\end{enumerate}
\end{assu}
Let $E$ be a finite totally ramified extension of $\mathrm{Frac} ( W(\mathbb{F}) )$ and write $\mathcal{O} := \mathcal{O}_E$ as before, and $E$ plays the role of the coefficient field and can be enlarged if necessary.


Following \cite[Page 717]{diamond-flach-guo} and \cite[Definition 4.6]{dimitrov-ihara}, we briefly recall the notion of $\Sigma$-ramified deformations of $\overline{\rho}$. See $\S$\ref{sec:deformation_theory} for the detailed description of local deformation conditions.

Let $\Sigma$ be a finite set of primes $\ell$ not equal to $p$. For a field $F$, $G_F$ denotes the absolute Galois group of $F$.
A deformation $\rho$ of $\overline{\rho}$ to a complete Noetherian local $\mathcal{O}$-algebra $R$ with residue field $\mathbb{F}$ is \textbf{$\Sigma$-ramified} if
 $\rho : G_{\mathbb{Q}} \to \mathrm{GL}_2(R)$ is a continuous representation such that 
\begin{itemize}
\item $\rho \vert_{G_{\mathbb{Q}_p}}$ is a low weight crystalline deformation in the sense of Definition \ref{defn:crys-low-wt},
\item  $\rho \vert_{G_{\mathbb{Q}_\ell}}$ is minimally ramified at $\ell \not\in \Sigma$ in the sense of Definition \ref{defn:minimally_ramified}, and
\item $\mathrm{det}(\rho) = \chi^{1-k}_{\mathrm{cyc}} \otimes_{\mathcal{O}} R$
\end{itemize} 
 where $\chi_{\mathrm{cyc}}$ is the $p$-adic cyclotomic character\footnote{The Hodge--Tate weight of the $p$-adic cyclotomic character is $-1$.}.
 A $\emptyset$-ramified deformation is called \textbf{minimally ramified} (in the global sense).
Denote by $\mathcal{R}_{\Sigma}$ the $\Sigma$-ramified deformation ring of $\overline{\rho}$.  


Let $S$ be a large finite set of primes containing $\Sigma$ and we define the \textbf{$\Sigma$-ramified Hecke algebra} $\mathcal{T}_{\Sigma}$ by the $\mathcal{O}$-subalgebra of $\prod_f \mathcal{O}$ generated by $\left( \iota_p ( a_\ell(f) ) \right)_{\ell \not\in S}$ where $f$ runs over newforms of fixed weight $k$ such that the associated $p$-adic Galois representation $\rho_f$ is a $\Sigma$-ramified deformation of $\overline{\rho}$. (e.g. \cite[$\S$3.3]{ddt}) Note that $\mathcal{T}_{\Sigma}$ is independent of $S$.
Following \cite[Definition 4.2]{dimitrov-ihara}, we recall
\begin{equation} \label{eqn:P_rhobar}
P_{\overline{\rho}} :=\left\lbrace \ell \textrm{ primes} :   \ell \vert N(\overline{\rho}), \ \overline{\rho}\vert_{G_{\mathbb{Q}_\ell}} \textrm{ is irreducible, }
\overline{\rho}\vert_{I_{\ell}}  \textrm{ is reducible, } \ell \equiv -1 \pmod{p}
 \right\rbrace
\end{equation}
where $I_\ell \subseteq G_{\mathbb{Q}_\ell}$ is the inertia subgroup.
\begin{rem}
Due to the argument in \cite[$\S$1.7.1]{diamond-flach-guo}, we may assume that all the newforms in this article have minimal conductor among its twist when we work with $\mathrm{ad}^0(f)$. (cf. \cite[$\S$4.1]{dimitrov-ihara}.)
\end{rem}
\subsubsection{The level structure $N_{\overline{\rho},\Sigma}$ for the Selmer computation} \label{subsubsec:level-str-Selmer}
Following \cite[(21)]{dimitrov-ihara}, we introduce the level structure depending on $\overline{\rho}$ and $\Sigma$.
For $\ell \neq p$, we put
$c(\ell) := \mathrm{ord}_{\ell} (N(\overline{\rho})  )$, $d(\ell) := \mathrm{dim}_{\mathbb{F}} \ \mathrm{H}^0(I_\ell, \overline{\rho})$,
and $c(p) = d(p) = 0$.
For $\overline{\rho}$ and $\Sigma$, we write
\begin{equation} \label{eqn:level_Sigma}
N_{\overline{\rho},\Sigma} := r \cdot \prod_{\ell \in \Sigma} \ell^{c(\ell) + d(\ell)} \cdot \prod_{\ell \not\in \Sigma} \ell^{c(\ell)}
 =r \cdot N(\overline{\rho}) \cdot \prod_{\ell \in \Sigma} \ell^{d(\ell)}
\end{equation}
where $r > 3$ is an auxiliary prime such that $r \not\equiv 1 \pmod{p}$, $\overline{\rho}$ is unramified at $r$, and 
$ \left( \mathrm{tr} \left( \overline{\rho} (\mathrm{Frob}_r) \right) \right)^2 \neq r^{k-2} (1+r)^2$ in $\mathbb{F}$.
Here, $\mathrm{Frob}_r$ is a geometric Frobenius at $r$. The prime $r$ is needed to have \emph{neat} levels.
It is known that there exist infinitely many such $r$ (e.g.~\cite[Lemma 11]{diamond-taylor-non-optimal}, \cite[Lemma 2 (when $p=3$)]{diamond-taylor-lifting}).

Enlarge $S$ enough to contain $\Sigma$, primes dividing $N(\overline{\rho})$, and $r$ as in \cite[$\S$5.3]{dimitrov-ihara}.
Let $\mathbb{T}^S = \mathcal{O}[T_\ell, S_\ell : \ell \not\in S]$ be the abstract Hecke algebra over $\mathcal{O}$ generated by
standard Hecke operators $T_\ell$ and $S_\ell$ for every prime $\ell$ not in $S$ (``$S$-anemic").
Denote by $\mathfrak{m}_{\overline{\rho}} = (\lambda, T_\ell - \mathrm{tr}(\overline{\rho} (\mathrm{Frob}_\ell) ),  S_\ell - \ell^{-1}\cdot \mathrm{det}(\overline{\rho} (\mathrm{Frob}_\ell )) : \ell \not\in S)$ the maximal ideal of $\mathbb{T}^S$ corresponding to $\overline{\rho}$.

Let $\mathcal{F}^k_p = \mathrm{Sym}^{k-2}(R^1s_*\mathbb{Z}_p)$ is the $p$-adic local system where $s$ is the map from the universal elliptic curve to the modular curve with the $\Gamma_1(N)$-level  structure over $\mathrm{Spec}(\mathbb{Z}[1/N])$ with $k\geq 2$, $N \geq 3$ following \cite[$\S$1.2.3]{diamond-flach-guo}. See also \cite[(2) in $\S$2.1]{dimitrov-ihara}.

We recall the Hecke modules and Hecke algebras following \cite[(35)]{dimitrov-ihara}.

For an integer $N \geq 1$, denote by $\mathbb{T}(N)$ the full Hecke algebra over $\mathcal{O}$ faithfully acting on the Hecke module
$\mathrm{H}^1_{\mathrm{\acute{e}t},c}(Y_{1}(N), \mathcal{F}^k_p)[\langle-\rangle - \mathbf{1}]$ where $\mathrm{H}^1_{\mathrm{\acute{e}t},c}$ means the compactly supported \'{e}tale cohomology, and $M[\langle-\rangle - \mathbf{1}]$ is the Hecke submodule of $M$ on which the diamond operator acts trivially.

Let $\mathfrak{m}_\Sigma \subseteq \mathbb{T}(N_{\overline{\rho},\Sigma})$ be the maximal ideal generated by
$\mathfrak{m}_{\overline{\rho}}$, $U_r - \alpha_r$ and, $U_q$ for $q \in \Sigma$ where $\alpha_r$ is a chosen eigenvalue of $\overline{\rho}(\mathrm{Frob}_r)$.

Denote by $\mathbb{T}(N_{\overline{\rho},\Sigma})_{\mathfrak{m}_\Sigma}$ the localization of $\mathbb{T}(N_{\overline{\rho},\Sigma})$ at $\mathfrak{m}_\Sigma$.
As in \cite[$\S$5.3]{dimitrov-ihara}, $\mathbb{T}(N_{\overline{\rho},\Sigma})_{\mathfrak{m}_\Sigma}$ can be identified with the image of $\mathbb{T}^S$ in the ring of $\mathcal{O}$-endomorphisms of 
\[
\mathcal{M}_{N_{\overline{\rho},\Sigma}} := \mathrm{H}^1_{\mathrm{\acute{e}t},c}(Y_{1}(N_{\overline{\rho},\Sigma}), \mathcal{F}^k_p)[\langle-\rangle - \mathbf{1}]_{ \mathfrak{m}_\Sigma } .
\]
See  \cite[Lemma 5.4.(iii)]{dimitrov-ihara} and \cite[Proposition 2.4.2]{epw} for details.

\begin{rem} \label{rem:eichler-shimura-isom}
Denote by $S_k(\Gamma_0(N), \mathcal{O})$ the $\mathcal{O}$-module generated by \emph{normalized} eigenforms in $S_k(\Gamma_0(N))$.
Although it is easy to observe that
$\mathcal{M}_{N_{\overline{\rho},\Sigma}}$ is isomorphic to two copies of $S_k(\Gamma_0(N_{\overline{\rho},\Sigma}), \mathcal{O})_{\mathfrak{m}_\Sigma}$ 
via the \emph{integral} Eichler--Shimura isomorphism \cite[(3) in $\S$1.2]{vatsal-cong} under Assumption \ref{assu:working_assumptions},
we do not use the $\Gamma_0(N_{\overline{\rho},\Sigma})$-level structure for $\mathcal{M}_{N_{\overline{\rho},\Sigma}}$  directly in order to avoid the issues on the moduli problem and the smoothness.
Notably, $\mathbb{T}(N_{\overline{\rho},\Sigma})_{\mathfrak{m}_\Sigma}$ acts faithfully on $S_k(\Gamma_0(N_{\overline{\rho},\Sigma}))_{\mathfrak{m}_\Sigma}$.
\end{rem}
\begin{lem} \label{lem:sigma-level-identification}
If $\Sigma$ contains $P_{\overline{\rho}}$, then
there exists a unique isomorphism of $\mathbb{T}^S$-algebras $\mathbb{T}(N_{\overline{\rho},\Sigma})_{\mathfrak{m}_\Sigma} \simeq \mathcal{T}_{\Sigma}$.
\end{lem}
\begin{proof}
As in \cite[Lemma 5.4]{dimitrov-ihara},
 it is enough to show that there exists a unique isomorphism of $\mathbb{T}^S\otimes_{\mathcal{O}} \mathbb{C}$-algebras $\mathbb{T}(N_{\overline{\rho},\Sigma})\otimes_{\mathcal{O}} \mathbb{C}\simeq \mathcal{T}_{\Sigma}\otimes_{\mathcal{O}} \mathbb{C}$,
 where we fix an embedding $E \to \mathbb{C}$.
Since our $\mathcal{M}_{N_{\bar{\rho} }, \Sigma}\otimes \C$ is isomorphic to $\mathcal{M}_{\Sigma}\otimes \C$ in \cite{dimitrov-ihara} with $F=\Q$.
Hence, by \cite[Lemma 6.4.(i)]{dimitrov-ihara}, we complete the proof.
\end{proof}
\begin{rem}
Note that $P_{\overline{\rho}} \subseteq \Sigma$ is not assumed in \cite[Proposition 4.7]{ddt} since this case is excluded there. See \cite[Remark 3.7 and $\S$7.2]{diamond_extension_wiles} for detail. It corresponds to the type \textbf{V} in \cite[$\S$2]{diamond_extension_wiles}.
In \cite{diamond-flach-guo}, it is bypassed by using \cite[Lemma 1.5]{diamond-flach-guo}. See also the comment right after
\cite[Proposition 6.5]{dimitrov-ihara}.
\end{rem}

\subsubsection{The level structure $N^{\Sigma}_f$ for the $L$-value computation}
Following \cite[$\S$1.7.3]{diamond-flach-guo}, we introduce another level structure depending on a newform $f$ of level $N_f$ and $\Sigma$.
Let $d_0(\ell) := \mathrm{dim}_{E} \ \mathrm{H}^0(I_\ell, V_f)$ and define
\begin{equation} \label{eqn:level^Sigma}
N^{\Sigma}_f := N_f \cdot \prod_{\ell \in \Sigma} \ell^{d_0(\ell)} .
\end{equation}
For a newform $f = \sum a_n q^n$ of level $N_f$ and $\Sigma$,
we define the $\Sigma$-imprimitive eigenform $f^{\Sigma} = \sum b_n q^n$ of level $N^{\Sigma}_f$ by
$$b_n = 
\left\lbrace \begin{array}{ll}
0 & \textrm{ if $n$ is divisible by a prime in $\Sigma$, } \\
 a_n & \textrm{ otherwise. } 
 \end{array} \right.$$
Let $\mathbb{T} (N^{\Sigma}_f) $ be the full Hecke algebra over $\mathcal{O}$ faithfully acting on 
$\mathrm{H}^1_{\mathrm{\acute{e}t},c}(Y_{1}(N^{\Sigma}_f), \mathcal{F}^k_p)[\langle-\rangle - \mathbf{1}]$, and
$\mathfrak{m}^\Sigma \subseteq \mathbb{T} (N^{\Sigma}_f) $ be the maximal ideal generated by
$\mathfrak{m}_{\overline{\rho}}$ and $U_q$ for $q \in \Sigma$.
Denote by $\mathbb{T} (N^{\Sigma}_f)_{\mathfrak{m}^\Sigma}$ the localization of $\mathbb{T} (N^{\Sigma}_f)$ at $\mathfrak{m}^\Sigma$, and it is isomorphic to
 the image of $\mathbb{T}^S$ in $\mathrm{End}_{ \mathcal{O} } ( \mathrm{H}^1_{\mathrm{\acute{e}t},c}(Y_{1}(N^{\Sigma}_f), \mathcal{F}^k_p)[\langle-\rangle - \mathbf{1}]_{ \mathfrak{m}^\Sigma } )$ as before.

\subsubsection{Congruence ideals and the $R=\mathbb{T}$ theorem}
\begin{defn}[Congruence ideals; \cite{diamond-euler, diamond-flach-guo, dimitrov-ihara}] 
\begin{enumerate}
\item
For a newform $f$ which arises as a $\Sigma$-ramified deformation of $\overline{\rho}$, 
let 
$\pi_{f,\Sigma} : \mathcal{T}_\Sigma \to \mathcal{O}$
be the projection to the $f$-component.
Then the \textbf{$\Sigma$-ramified congruence ideal of $f$} is defined by
$$\eta_{f, \Sigma} := \pi_{f, \Sigma} \left( \mathrm{Ann}_{\mathcal{T}_\Sigma}( \mathrm{ker} \ \pi_{f, \Sigma} ) \right).$$
\item
For a newform $f \in S_k(\Gamma_0(N))$ and $\Sigma$, consider the $\Sigma$-imprimitive eigenform $f^\Sigma$ of level $N^{\Sigma}_f$ associated to $f$ 
and define
$\pi_{f^\Sigma} : \mathbb{T}(N^\Sigma_f)_{\mathfrak{m}^\Sigma} \to \mathcal{O}$
 by $T_\ell \mapsto a_\ell(f)$ for all $\ell \nmid N^\Sigma_f$.
Then the \textbf{congruence ideal of $f^\Sigma$} is defined by
$$\eta_{f^\Sigma}(N^\Sigma_f) := \pi_{f^\Sigma} \left( \mathrm{Ann}_{\mathbb{T}(N^\Sigma_f)_{\mathfrak{m}^\Sigma}}( \mathrm{ker} \ \pi_{f^\Sigma} ) \right) .$$
\end{enumerate}
\end{defn}
We also define the eigenform $f^{\Sigma, \alpha_r}$ of level $N_{\overline{\rho}, \Sigma}$ by
the $r$-stabilization of $f^\Sigma$ with $U_r$-eigenvalue $\alpha_r$.
\begin{lem}  \label{lem:removing_level_r} 
\begin{enumerate}
\item 
If we choose $\Sigma$ by the set of primes dividing $N_f/N(\overline{\rho})$, then $N_{\overline{\rho},\Sigma} = r \cdot N^\Sigma_f$.
\item
If $\Sigma$ contains $P_{\overline{\rho}}$ and $N_{\overline{\rho}, \Sigma} = r \cdot N^{\Sigma}_f$, then we have
$\eta_{f^\Sigma}(N^\Sigma_f) = \eta_{f^{\Sigma, \alpha_r}}(N_{\overline{\rho},\Sigma}) = \eta_{f, \Sigma} $
for a newform $f$ of level $N$ dividing $N_{\overline{\rho},\Sigma}$.
\end{enumerate}
\end{lem}
\begin{proof}
Since the first statement immediately follows from the definitions, we focus on the second one.
Since the second equality immediately follows from Lemma \ref{lem:sigma-level-identification}, it suffices to check the first equality.
By a basic property of the congruence ideals (cf. \cite[(5.2.2)]{ddt}), we have
$\eta_{f^{\Sigma \cup \lbrace r \rbrace}}(r^2 \cdot N^\Sigma_f) \subseteq \eta_{f^{\Sigma, \alpha_r}}(r \cdot N^\Sigma_f)  \subseteq \eta_{f^\Sigma}(N^\Sigma_f) $.
It implies that
$ \#\mathcal{O} / \eta_{f^{\Sigma \cup \lbrace r \rbrace}}(r^2 \cdot N^\Sigma_f) \geq
\#\mathcal{O} / \eta_{f^{\Sigma, \alpha_r}}(r \cdot N^\Sigma_f)  \geq 
\#\mathcal{O} /\eta_{f^\Sigma}(N^\Sigma_f) $.
The properties of $r$ directly implies that the Euler factor of the adjoint $L$-function of $f$ at $s=1$ is a unit.
By the freeness of the Hecke modules (e.g. $\S$\ref{subsec:cohomology_congruence_ideals}), \cite[Proposition 6.3 and Proof of Theorem 6.6.(1)]{dimitrov-ihara}, and the property of $r$ above, we have
$ \#\mathcal{O} / \eta_{f^{\Sigma \cup \lbrace r \rbrace}}(r^2 \cdot N^\Sigma_f) =
\#\mathcal{O} /\eta_{f^\Sigma}(N^\Sigma_f) $.
Therefore, the conclusion follows.
\end{proof}
\begin{rem}
The level $N_{\overline{\rho},\Sigma}$ in (\ref{eqn:level_Sigma}) is useful to work with $R=\mathbb{T}$ theorem and 
the level $N^\Sigma_f$ in (\ref{eqn:level^Sigma}) is convenient when we work with adjoint $L$-values.
\end{rem}
We are now ready to recall the $R=\mathbb{T}$ theorem which we use later.
\begin{thm}[Diamond--Flach--Guo \cite{diamond-flach-guo}, Dimitrov \cite{dimitrov-ihara}] \label{thm:diamond-flach-guo-dimitrov}
Let $\overline{\rho} : G_{\mathbb{Q}} \to \mathrm{GL}_2(\mathbb{F})$ be an odd, absolutely irreducible, and continuous Galois representation satisfying Assumption \ref{assu:working_assumptions}. Assume that $\Sigma$ contains $P_{\overline{\rho}}$.
Then the natural surjective map
$$\pi_\Sigma : \mathcal{R}_{\Sigma} \to \mathcal{T}_{\Sigma}$$
 is an isomorphism of finite flat complete intersections over $\mathcal{O}$ and 
$\mathcal{M}_{N_{\overline{\rho},\Sigma}}$ is free of rank two over $\mathcal{T}_{\Sigma}$.
In particular, all $\Sigma$-ramified deformations of $\overline{\rho}$ are modular.\\
Moreover, for all newforms $f$ such that $\rho_f$ is a $\Sigma$-ramified deformation of $\overline{\rho}$, we have
$$\# \mathrm{Sel}_{\Sigma}(\mathbb{Q}, \mathrm{ad}^0(f) \otimes E/\mathcal{O}) = \# \mathcal{O}/ \eta_{f,\Sigma} < \infty$$
where $\mathrm{Sel}_{\Sigma}(\mathbb{Q}, \mathrm{ad}^0(f) \otimes E/\mathcal{O})$ is the $\Sigma$-imprimitive adjoint Selmer group defined in Definition \ref{defn:adjoint_selmer}.
\end{thm}
\begin{rem} 
The notion of  \emph{cohomology congruence ideals} is used in the original statement of \cite{diamond-flach-guo}. See $\S$\ref{subsec:cohomology_congruence_ideals} for detail. The term was coined by Diamond in \cite{diamond-euler}.
 Careful readers will notice that the modularity lifting theorems in \cite{diamond-flach-guo} and \cite{dimitrov-ihara} have slightly different conditions on the image of the residual representation. However, it is easy to check either one is enough to obtain the same result when the base field is $\mathbb{Q}$.
\end{rem}

Let $f \in S_k(\Gamma_0(N))$ be a normalized eigenform whose residual representation is isomorphic to $\overline{\rho}$. 
When we write $\mathbb{T}(N)_{\mathfrak{m}_f}$, it is the localization of the full Hecke algebra $\mathbb{T}(N)$ at the maximal ideal $\mathfrak{m}_f$ generated by the Hecke eigensystem of $f$ and $\lambda$.
As before, when $S$ contains $\Sigma$ and all prime divisors of $N$, 
$\mathbb{T}(N)_{\mathfrak{m}_f}$ can be identified with the image of $\mathbb{T}^S$ in $\mathrm{End}_{\mathcal{O}} \left(  \mathrm{H}^1_{\mathrm{\acute{e}t},c}(Y_{1}(N), \mathcal{F}^k_p)[\langle-\rangle - \mathbf{1}]_{\mathfrak{m}_f} \right)$.
Note that $\mathfrak{m}_f$ is  generated by $\mathfrak{m}_{\overline{\rho}}$, $U_q -\alpha_q(f)$ for $q$ exactly dividing $N$, and $U_q$ for $q$ dividing $N$ more than twice where  $\alpha_q$ is the $q$-th Fourier coefficient of $f$.
Since $\lambda \in \mathfrak{m}_f$, the choice of a congruent eigenform does not change the maximal ideal $\mathfrak{m}_f$.

We decompose $N = N^+ \cdot N^-$ with $(N^+, N^-) = 1$ and $N^-$ square-free.
Assume that $f$ is new at all primes dividing $N^-$.
Then the map
$\pi_{f}: \mathbb{T}(N^+N^-)_{\mathfrak{m}_f} \to \mathcal{O} $ defined by $T_\ell \mapsto a_\ell(f)$ for all $\ell \nmid N^+N^-$ factors through the $N^-$-new quotient $\mathbb{T}(N^+N^-)^{N^-}_{\mathfrak{m}_f}$ of $\mathbb{T}(N^+N^-)_{\mathfrak{m}_f}$.
Thus, we naturally have the surjective map
$\pi^{N^-}_{f} : \mathbb{T}(N^+N^-)^{N^-}_{\mathfrak{m}_f} \to \mathcal{O} $.
\begin{defn}
The \textbf{$N^-$-new congruence ideal of $f$} is defined by
$$\eta_{f}(N^+, N^-)  := \pi^{N^-}_{f} \left( \mathrm{Ann}_{\mathbb{T}(N^+N^-)^{N^-}_{\mathfrak{m}_f} }( \mathrm{ker} \ \pi^{N^-}_{f} ) \right) .$$
\end{defn}
Note that
$\mathbb{T}(N^+N^-)^{N^-}_{\mathfrak{m}_f}$ acts faithfully on the $N^-$-new subspace $S_k(\Gamma_0(N))^{N^-}_{\mathfrak{m}_f}$ of $S_k(\Gamma_0(N))_{\mathfrak{m}_f}$.

\subsection{Tamagawa exponents} \label{subsec:tamagawa_exponents}
Since we assume the image of $\overline{\rho}$ is irreducible (Assumption \ref{assu:working_assumptions}.(TW)), 
a Galois stable $\mathcal{O}$-lattice $T_f$ is uniquely determined up to scalars, so $\rho_n$ and $\overline{\rho}$ also have the same uniqueness property.
We define $A_{f} := V_{f}/ T_{f}$ and then have $A_{f}[\lambda^n] \simeq T_f/\lambda^n T_f$ for all $n \geq 1$.
Following \cite[Definition 3.3]{pw-mu}, Tamagawa exponents are defined as follows.
\begin{defn} \label{defn:tamagawa_exponents}
The \textbf{Tamagawa exponent $t_{f} (q)$ for $f$} is defined by the largest integer $t$ such that
$A_{f}[\lambda^{t}]$ is unramified at $q$ and $A_{f}[\lambda^{t+1}]$ is ramified  at $q$.
\end{defn}
Note that $t_{f} ( q )$ is finite if $q$ divides $N$, and $t_f(q) = 0$ if $q$ divides $N(\overline{\rho})$.
It is known that $(\lambda^{t_{f} (q)}) = \mathrm{Fitt}_{ \mathcal{O} } \left( \left( \mathrm{H}^1(I_{q^2}, T_{f} )^{G_{\mathbb{Q}_{q^2}}}\right)_{\mathrm{tors}} \right)$ when $q$ divides $N^-$ (c.f. \cite[I.4.2.2]{fontaine-perrin-riou}). Here, $\mathbb{Q}_{q^2}$ is the unramified quadratic extension of $\mathbb{Q}_{q}$ and $I_{q^2}$ is the inertia subgroup of $G_{\mathbb{Q}_{q^2}}$.
In the case of elliptic curves, the Tamagawa exponent also coincides with the $p$-exponent of the local Tamagawa factor at $q$ of the elliptic curve over $\mathbb{Q}_{q^2}$. 
See \cite[Page 1354]{pw-mu}, \cite[Corollary 2 and Corollary 6.15]{chida-hsieh-main-conj}, and \cite[$\S$6.3]{wei-zhang-mazur-tate}.

\subsection{The idea of proof}
In the proof of the Bloch--Kato conjecture for adjoint motives of modular forms \cite{diamond-flach-guo}, the following connections are established:
\[
\xymatrix{
\textrm{adjoint $L$-values} \ar@{<->}[r]^-{\textrm{Hida}} & { \substack{ \textrm{cohomology} \\
\textrm{congruence ideals} } } \ar@{<->}[r]^-{\textrm{freeness}}
& \textrm{congruence ideals} \ar@{<->}[r]^-{R= \mathbb{T}} & \textrm{adjoint Selmer groups} .
}
\]
\begin{itemize}
\item The connection between adjoint $L$-values and cohomology congruence ideals is due to Hida's formula \cite{hida-invent-1981}. See also \cite{ddt}.
\item The connection between cohomology congruence ideals and congruence ideals follows from the freeness of the Hecke module over the associated Hecke algebra. Such freeness results can be obtained in two ways: the method of Mazur, Ribet, Wiles, and Faltings--Jordan \cite{faltings-jordan} (cf. \cite[Theorem 1.13]{vatsal-cong}) based on the $q$-expansion principle and the crystalline comparison isomorphism (Fontaine's $C_{\mathrm{cris}}$) and the Diamond's improvement of the Taylor--Wiles system argument \cite{diamond_multiplicity_one}.
 In Diamond's method, the level of the Hecke module should be of the form $N_{\overline{\rho},\Sigma}$ for some $\Sigma$ \cite[Theorem 2.4]{diamond_multiplicity_one}.
\item The connection between congruence ideals and adjoint Selmer groups follows from the Taylor--Wiles system argument.
\end{itemize}
Our main theorem (Theorem \ref{thm:main_thm}) precisely measures the difference of two congruence ideals $\eta_f(N)$ and $\eta_f(N^+,N^-)$, but making such a connection is not straightforward at all. The following diagram summarizes how the connection is made.
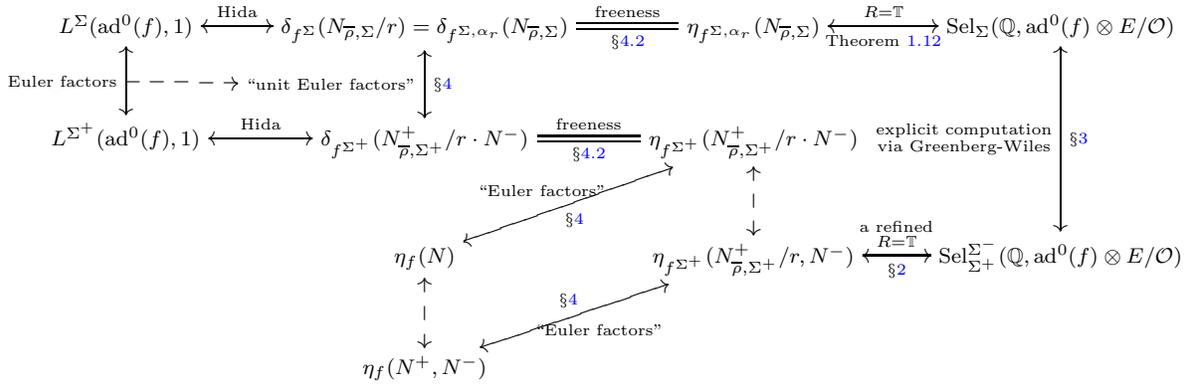
\[
{ \scriptsize
\xymatrix{
L^{\Sigma}(\mathrm{ad}^0(f), 1) \ar@{<->}[d]_-{\textrm{Euler factors}}="a" \ar@{<->}[r]^-{\textrm{Hida}} & \delta_{f^{\Sigma}}(N_{\overline{\rho},\Sigma}/r)= \delta_{f^{\Sigma, \alpha_r}}(N_{\overline{\rho},\Sigma}) \ar@{=}[r]^-{\textrm{freeness}}_-{\textrm{$\S$\ref{subsec:cohomology_congruence_ideals}}} \ar@{<->}[d]^-{\textrm{$\S$\ref{sec:adjoint_L-values}}}_-{``\textrm{unit Euler factors}"}="b" & \eta_{f^{\Sigma, \alpha_r}}(N_{\overline{\rho},\Sigma}) \ar@{<->}[r]^-{R=\mathbb{T}}_-{ \textrm{ Theorem \ref{thm:diamond-flach-guo-dimitrov} } } & \mathrm{Sel}_{\Sigma}(\mathbb{Q}, \mathrm{ad}^0(f) \otimes E/\mathcal{O}) \ar@{<->}[dd]_-{ { \substack{ \textrm{explicit computation} \\ \textrm{via Greenberg-Wiles} } } }^-{ \textrm{$\S$\ref{sec:adjoint_Selmer} } } \\
L^{\Sigma^+}(\mathrm{ad}^0(f), 1) \ar@{<->}[r]^-{\textrm{Hida}} & \delta_{f^{\Sigma^+}}(N^+_{\overline{\rho},\Sigma^+}/r \cdot N^-) \ar@{=}[r]^-{\textrm{freeness}}_-{\textrm{$\S$\ref{subsec:cohomology_congruence_ideals}}} & \eta_{f^{\Sigma^+}}(N^+_{\overline{\rho},\Sigma^+}/r \cdot N^-) \ar@{<-->}[d] \ar@{<->}[dl]_-{``\textrm{Euler factors}"}^-{\textrm{$\S$\ref{sec:adjoint_L-values}}} &  \\
 & \eta_{f}(N)  \ar@{<-->}[d] & \eta_{f^{\Sigma^+}}(N^+_{\overline{\rho},\Sigma^+} /r , N^-) \ar@{<->}[r]^-{ { \substack{ \textrm{a refined} \\ R = \mathbb{T} }} }_-{\textrm{$\S$\ref{sec:deformation_theory}}} \ar@{<->}[dl]^-{``\textrm{Euler factors}"}_-{\textrm{$\S$\ref{sec:adjoint_L-values}}} & \mathrm{Sel}^{\Sigma^-}_{\Sigma^+}(\mathbb{Q}, \mathrm{ad}^0(f) \otimes E/\mathcal{O}) \\
 &  \eta_{f}(N^+,N^-)
\ar@{-->}"a";"b"
}
}
\]
The proof of Theorem \ref{thm:main_thm} consists of two parts.
The first part (the Selmer computation) is an approximation of the main theorem (Theorem \ref{thm:key_thm}) whose proof is based on the $R=\mathbb{T}$ argument and the computation of Galois cohomology.
However, it does not give the exact formula but a slightly different formula.
The second one (the $L$-value computation) removes the difference (Proposition \ref{prop:second_correction}) whose proof is based on the explicit comparison among adjoint $L$-values.
Theorem \ref{thm:main_thm} immediately follows from these two results.
\subsubsection{The Selmer computation} \label{subsubsec:first-approximation}
Let $\Sigma$ be the set of primes dividing $N/N(\overline{\rho})$, $\Sigma^+$ the subset of $\Sigma$ consisting of primes not dividing $N^-$, and $\Sigma^- := \Sigma \setminus \Sigma^+$.
We also decompose $N(\overline{\rho}) = N(\overline{\rho})^+ \cdot N(\overline{\rho})^-$ following the decomposition $N = N^+ \cdot N^-$ in Theorem \ref{thm:main_thm} such that $N(\overline{\rho})^{\pm} \vert N^{\pm}$, respectively.
For a prime $\ell$ dividing $N^+$, we put
$c^+(\ell) := \mathrm{ord}_{\ell} (N(\overline{\rho})^+  )$ and $d(\ell) = \mathrm{dim}_{\mathbb{F}} \ \mathrm{H}^0(I_\ell, \overline{\rho})$. We define
\begin{equation} \label{eqn:level_N^+}
N^+_{\overline{\rho},\Sigma^+} := r \cdot \prod_{\ell \in \Sigma^+} \ell^{c^+(\ell) + d(\ell)} \prod_{\ell \not\in \Sigma^+, \ell \vert N(\overline{\rho})^+} \ell^{c^+(\ell)}
\end{equation}
where $r$ is the same one in (\ref{eqn:level_Sigma}).
Then we have $N^+ \vert N^+_{\overline{\rho},\Sigma^+}$. Note that $\ell \in \Sigma^+$ implies $\ell^2 \vert N^+_{\overline{\rho},\Sigma^+}$.
Comparing with $N(\overline{\rho})^+$ and $N^+$, we have
\begin{align*}
N^+_{\overline{\rho},\Sigma^+} & = r \cdot  N(\overline{\rho})^+ \cdot \prod_{\ell \in \Sigma^+} \ell^{d(\ell)} \\
& = r \cdot  N^+ \cdot \prod_{\ell \in \Sigma^+} \ell^{d_0(\ell)} .
\end{align*}
How much $N^+_{\overline{\rho},\Sigma^+}$ and $N^+$ differ at primes in $\Sigma^+$?
For $\ell \in \Sigma^+$, we observe the following:
\begin{itemize}
\item $d_0(\ell) \neq 2$ since $\Sigma^+$ is (minimally) chosen;
\item $d_0(\ell) = 0$ if $\ell^2$ divides $N^+$; 
 \item $d_0(\ell) = 1$ if $\ell$ divides $N^+$ exactly.
\end{itemize}
Note that $\mathrm{ord}_\ell (N^+_{\overline{\rho},\Sigma^+}/N^+) \neq 0$ happens only in the last case. Due to the minimal choice of $\Sigma$, this case is equivalent to the case $\ell \Vert N^+$ but $\ell \nmid N(\overline{\rho})$.

\begin{thm}[Selmer computation] \label{thm:key_thm}
Let $f \in S_k(\Gamma_0(N))$ be a newform with $\overline{\rho}_f \simeq \overline{\rho}$, and
$\Sigma$ the set of primes dividing $N/N(\overline{\rho})$.
Suppose that $\overline{\rho}$ satisfies Assumption \ref{assu:working_assumptions}.
We assume the following conditions:
\begin{itemize}
\item $2 \leq k \leq p-1$.
\item $\Sigma$ contains $P_{\overline{\rho}}$ (defined in (\ref{eqn:P_rhobar})).
\item $p$, $N^+$, and $N^-$ are pairwisely relatively prime.
\item $N^-$ is square-free.
\item For a prime divisor $q$ of $N^-$, if $q \equiv  1 \pmod{p}$, then $\overline{\rho}$ is ramified at $q$.
\end{itemize}
Then we have
\begin{equation} \label{eqn:first_approximation}
\mathrm{ord}_\lambda \eta_{f^{\Sigma}}(N^{\Sigma}_f) = \mathrm{ord}_\lambda \eta_{f^{\Sigma^+}}(N^+_{\overline{\rho},\Sigma^+}/r, N^-) +  \sum_{q \vert N^-} t_f(q) .
\end{equation}
\end{thm}
We prove Theorem \ref{thm:key_thm} in $\S$\ref{sec:deformation_theory} and $\S$\ref{sec:adjoint_Selmer}. 
In $\S$\ref{sec:deformation_theory}, we review the deformation theory of Galois representations, study a presentation of the Galois deformation ring, and prove a refined $R=\mathbb{T}$ theorem.
In $\S$\ref{sec:adjoint_Selmer}, we review the standard facts of Galois cohomology and compute the difference between adjoint Selmer groups with different local conditions.
\begin{rem} 
\begin{enumerate}
\item Lemma \ref{lem:removing_level_r} is used in (\ref{eqn:first_approximation}).
\item 
\emph{The disadvantage} of the Selmer computation is the rigidity of the level structure $N^+_{\overline{\rho},\Sigma^+}/r$.
If $\ell \in \Sigma^+$, then $\ell^2$ must divide $N^+_{\overline{\rho},\Sigma^+}$.
For example, the Selmer computation does not imply
$$\mathrm{ord}_\lambda \eta_{f^{\Sigma}}(N^{\Sigma}_f)
 = \mathrm{ord}_\lambda \eta_{f^{\Sigma^+}}(N^+_{\overline{\rho},\Sigma^+}/r \cdot \ell , N^-/ \ell) +  \sum_{q \vert N^-/\ell} t_f(q) $$
for any prime $\ell$ dividing $N^-$ at which $\overline{\rho}$ is unramified (cf. \cite{lundell-level-lowering}).
\end{enumerate}
\end{rem}
In order to obtain Theorem \ref{thm:main_thm} from Theorem \ref{thm:key_thm}, we will show the following implications
\begin{align*}
 \mathrm{ord}_\lambda \eta_{f^{\Sigma}}(N^{\Sigma}_f) & = \mathrm{ord}_\lambda \eta_{f^{\Sigma^+}}(N^+_{\overline{\rho},\Sigma^+}/r, N^-) +  \sum_{q \vert N^-} t_f(q)  \\
 \Rightarrow
\mathrm{ord}_\lambda \eta_{f^{\Sigma}}(N^+_{\overline{\rho},\Sigma^+}/r \cdot N^-) & = \mathrm{ord}_\lambda \eta_{f^{\Sigma^+}}(N^+_{\overline{\rho},\Sigma^+}/r, N^-) +  \sum_{q \vert N^-} t_f(q) \\
 \Rightarrow
\mathrm{ord}_\lambda \eta_{f^{\Sigma}}(N^+ \cdot N^-) & = \mathrm{ord}_\lambda \eta_{f^{\Sigma^+}}(N^+, N^-) +  \sum_{q \vert N^-} t_f(q) 
\end{align*}
by using the $L$-value computation.

\subsubsection{The $L$-value computation} \label{subsubsec:second_correction}
We assume the freeness result described in $\S$\ref{subsec:cohomology_congruence_ideals} to identify the congruence ideals and the cohomology congruence ideals here.
The $L$-value computation is the following proposition.
\begin{prop}[$L$-value computation] \label{prop:second_correction}
We keep the assumptions of Theorem \ref{thm:key_thm}.
\begin{enumerate}
\item 
Let $\Sigma^-$ is the set of primes dividing $N^-$ where $\overline{\rho}$ is unramified.
If any prime in $\Sigma^-$ is not congruent to $\pm 1$ modulo $p$, then
$$\mathrm{ord}_\lambda \eta_{f^{\Sigma}}(N^{\Sigma}_f)  = \mathrm{ord}_\lambda \eta_{f^{\Sigma^+}}(N^+_{\overline{\rho},\Sigma^+}/r \cdot N^-).$$
\item  We have
$$\mathrm{ord}_\lambda \eta_{f^{\Sigma^+}}(N^+_{\overline{\rho},\Sigma^+} / r \cdot N^-) - \mathrm{ord}_\lambda \eta_{f}(N) = \mathrm{ord}_\lambda \eta_{f^{\Sigma^+}}(N^+_{\overline{\rho},\Sigma^+} / r , N^-)  - \mathrm{ord}_\lambda \eta_{f}(N^+, N^-) .$$
\end{enumerate}
\end{prop}
We prove Proposition \ref{prop:second_correction} in $\S$\ref{sec:adjoint_L-values}.
In $\S$\ref{sec:adjoint_L-values}, we recall cohomology congruence ideals and study their interpretation as the adjoint $L$-values.
Theorem \ref{thm:main_thm} immediately follows from Theorem \ref{thm:key_thm} and Proposition \ref{prop:second_correction}.

\section{Deformation theory and a refined $R=\mathbb{T}$ theorem} \label{sec:deformation_theory}
We recall the relevant deformation theory of Galois representations and prove a refined $R=\mathbb{T}$ theorem (Theorem \ref{thm:refined_R=T}).

Let $\overline{\rho}$ be the residual representation fixed in $\S$\ref{subsec:modularity_lifting_congruence_ideals}.
Then $\overline{\rho}$ factors through $G_{\mathbb{Q}, S} = \mathrm{Gal}(\mathbb{Q}_S/\mathbb{Q})$ where
$\mathbb{Q}_S$ is the maximal extension of $\mathbb{Q}$ unramified outside $S \cup \{\infty\}$.
From now on, we regard $\overline{\rho}$ as a representation of $G_{\mathbb{Q}, S}$
$$\overline{\rho} : G_{\mathbb{Q}, S} \to \mathrm{GL}_2(\mathbb{F}).$$ 
We fix the determinant of all the liftings and deformations of $\overline{\rho}$ to be $\chi^{1-k}_{\mathrm{cyc}}$ throughout this article.

\subsection{A local-global principle of deformation functors} \label{subsec:local-global-deformation}
Let $\mathrm{CNL}_{\mathcal{O}}$ be the category of complete Noetherian local (CNL) $\mathcal{O}$-algebras with residue field $\mathbb{F}$ whose morphisms are local $\mathcal{O}$-algebra morphism inducing the identity map on $\mathbb{F}$.
Let $\mathfrak{D}_{G_{\mathbb{Q},S}}: \mathrm{CNL}_{\mathcal{O}} \to \mathrm{Sets}$ be the functor such that
for $R\in \mathrm{CNL}_{\mathcal{O}}$,
 $\mathfrak{D}_{G_{\mathbb{Q},S}}(R)$
is the set of deformations $\rho: G_{\mathbb{Q},S}\to \mathrm{GL}_2(R)$ of $\overline{\rho}$
such that $\det(\rho)=\chi^{1-k}_{\mathrm{cyc}}\otimes_{\mathcal{O}}R$.
In particular, $\rho$ is unramified outside $S  \cup \{\infty\}$.

 Under Assumption \ref{assu:working_assumptions}.(TW),
 $\mathfrak{D}_{G_{\mathbb{Q},S}}$ is representable and represented by the universal deformation
  ring 
 $\mathcal{R}_{G_{\mathbb{Q},S}}$ (with fixed determinant $\chi^{1-k}_{\mathrm{cyc}}$).
 If we consider $\Sigma$-ramified deformations, we impose unrestricted deformation conditions at primes in $\Sigma$.

We impose certain local conditions at $\ell \in S$ to cut out irrelevant deformations in $\mathcal{R}_{G_{\mathbb{Q},S}}$. 
Imposing these local conditions can be interpreted as defining relatively representable subfunctors $\mathfrak{D}(\ell)$ of the universal deformation functor $\mathfrak{D}_{G_{\mathbb{Q}_\ell}}$ which classifies the deformations of $\overline{\rho}\vert_{G_{\mathbb{Q}_\ell}}$ with determinant $\chi^{1-k}_{\mathrm{cyc}}$ (see \cite[$\S$19]{mazur-intro-deformation-theory} and \cite[$\S$3]{bockle-presentations} for the details on relatively representable functors, noting that $\mathcal{C}_{\mathcal{O}}$ in \cite{bockle-presentations} denotes our  $\mathrm{CNL}_{\mathcal{O}}$).
We note that since $\overline{\rho}\vert_{G_{\mathbb{Q}_\ell}}$ may not satisfy $\mathrm{End}_{\mathbb{F}} (\overline{\rho}\vert_{G_{\mathbb{Q}_\ell}}) = \mathbb{F}$, we will only have relatively representable functors for local deformations in general. 
Given $\mathfrak{D}(\ell) \subseteq \mathfrak{D}_{G_{\mathbb{Q}_\ell}}$ as above for $\ell \in S$,  
a subfunctor $\mathfrak{D}(S): \mathrm{CNL}_{\mathcal{O}} \to \mathrm{Sets}$ of $\mathfrak{D}_{G_{\mathbb{Q},S}}$ is defined by the cartesian diagram of functors
\[
\xymatrix{
\mathfrak{D}(S) \ar[r]  \ar[d] 
 & \prod_{\ell \in S}\mathfrak{D}(\ell) \ar[d] \\
\mathfrak{D}_{G_{\mathbb{Q},S}} \ar[r]^-{``\mathrm{res}"} &  \prod_{\ell \in S} \mathfrak{D}_{G_{\mathbb{Q}_\ell}},
}
\]
where $``\mathrm{res}"$ is the natural transformation arising from the restriction of $G_{\mathbb{Q},S}\to \mathrm{GL}_2(\mathbb{F})$ to $G_{\Q_\ell}\to \mathrm{GL}_2(\mathbb{F})$.
More explicitly, for $R\in \mathrm{CNL}_{\mathcal{O}}$, $\mathfrak{D}(S)(R)$ is the set of deformations $\rho : G_{\mathbb{Q},S} \to \mathrm{GL}_2(R)$ contained in $\mathfrak{D}_{G_{\mathbb{Q},S}}(R)$
such that the restriction $\rho\vert_{G_{\mathbb{Q}_\ell}} : G_{\mathbb{Q}_\ell} \to \mathrm{GL}_2(R)$ satisfies the local deformation condition imposed by $\mathfrak{D}(\ell)(R)$ for all $\ell \in S$ \cite[$\S$23]{mazur-intro-deformation-theory}.
 By \cite[Proposition 3.4]{bockle-presentations}, $\mathfrak{D}(S)$ is also representable, and the deformation ring corresponding to $\mathfrak{D}(S)$ is denoted by $\mathcal{R}(S)$.

The tangent space of the deformation functors can be described in terms of adjoint Selmer groups.
For $\ell \in S$, let $\rho_\ell \in \mathfrak{D}_{G_{\Q_\ell}}(R)$
 be a deformation of $\overline{\rho}\vert_{G_{\mathbb{Q}_\ell}}$ to a finite-length ring $R \in \mathrm{CNL}_{\mathcal{O}}$.
Then the \textbf{tangent space $t_{\rho_\ell}$ of  $\mathfrak{D}_{G_{\mathbb{Q}_\ell}}$ at $\rho_\ell \in \mathfrak{D}_{G_{\Q_\ell}}(R)$} is defined by the inverse image of $\lbrace \rho_\ell \rbrace$ under the map $\mathfrak{D}_{G_{\Q_\ell}}(R[\varepsilon]/(\varepsilon^2))\to \mathfrak{D}_{G_{\Q_\ell}}(R)$ induced by $\varepsilon\mapsto 0$. Then it naturally has a structure of $R$-module, and there exists a natural isomorphism 
\begin{equation}\label{eqn:isom-tangent}
t_{\rho_\ell} \cong  \mathrm{H}^1(\mathbb{Q}_\ell, \mathrm{ad}^0(\rho_\ell))
\end{equation}
 (cf. \cite[\S 21]{mazur-intro-deformation-theory}).
Thus, $\mathfrak{D}(\ell)$ gives a submodule $\mathfrak{D}(\ell)(R[\varepsilon]/(\varepsilon^2))\cap t_{\rho_\ell}$ of 
$t_{\rho_\ell},$
and via (\ref{eqn:isom-tangent}) it also give a subspace of $ \mathrm{H}^1(\mathbb{Q}_\ell, \mathrm{ad}^0(\rho_{\ell}))$
which depends on $\rho_\ell$.

Let $\rho: G_{\mathbb{Q},S } \to \mathrm{GL}_2(\mathcal{O}) \in \mathfrak{D}(S)(\mathcal{O}),$
and let $\pi_S : \mathcal{R}(S) \to \mathcal{O}$ be the corresponding specialization with kernel denoted by
 $\wp_{\pi_S}$.
Let $\rho_n : G_{\mathbb{Q},S} \to \mathrm{GL}_2(\mathcal{O}/ \lambda^n)$ be the mod $\lambda^n$ reduction of $\rho$.
The restriction of $\rho_n$ to $G_{\mathbb{Q}_\ell}$ yields a subgroup $L_{\ell,n} \subseteq \mathrm{H}^1(\mathbb{Q}_\ell, \mathrm{ad}^0(\rho_n))$ as explained above, and the direct limit defines $L_\ell \subseteq \mathrm{H}^1(\mathbb{Q}_\ell, \mathrm{ad}^0(\rho) \otimes E/\mathcal{O})$.
Then $\mathcal{L} := (L_\ell)_{\ell \in S}$ becomes a Selmer structure for $S$ and $\mathrm{ad}^0(\rho) \otimes E/\mathcal{O}$ and we have an isomorphism
\begin{equation} \label{eqn:tangent-selmer}
\mathrm{Hom}_{\mathcal{O}}( \wp_{\pi_S} / \wp^2_{\pi_S} , E/\mathcal{O} ) \simeq \mathrm{Sel}_{\mathcal{L}} (G_{\mathbb{Q},S}, \mathrm{ad}^0(\rho) \otimes E/\mathcal{O})
\end{equation}
as in \cite[Proposition 1.2]{wiles} and \cite[$\S$28]{mazur-intro-deformation-theory}.

Let $\mathcal{C}_{\mathcal{O}}$ be the category whose objects are pairs $(R, \pi_R)$ where $R$ is an object of $\mathrm{CNL}_{\mathcal{O}}$ and $\pi_R: R \to \mathcal{O}$ is a surjective local $\mathcal{O}$-algebra homomorphism, and whose morphisms are morphisms in $\mathrm{CNL}_{\mathcal{O}}$ which commute with $\pi_R$'s.
For a pair $(R, \pi_R)$ in $\mathcal{C}_{\mathcal{O}}$, write $\wp_{\pi_R} := \mathrm{ker} \ \pi_R$.
As indicated in (\ref{eqn:tangent-selmer}), the \textbf{cotangent space to $\mathrm{Spec} \ R$ at $\wp_{\pi_R}$} is defined by
$$\Phi_{\pi_R} := \wp_{\pi_R} / \wp^2_{\pi_R}$$
and
the \textbf{congruence ideal of $(R, \pi_R)$} by
$\eta_{\pi_R} := \pi_R \left( \mathrm{Ann}_{R}(\wp_{\pi_R}) \right) \subseteq \mathcal{O} $.

The relation between the cotangent space and the congruence ideal is now well-known as follows:
\begin{prop} \label{prop:complete_intersection_congruence_ideals}
Let $(R, \pi_R) \in \mathcal{C}_{\mathcal{O}}$ such that $R$ is a finite flat $\mathcal{O}$-algebra and suppose that $\eta_R \neq 0$.
Then $R$ is a complete intersection if and only if $\mathrm{Fitt}_{\mathcal{O}} ( \Phi_{\pi_R} ) = \eta_{\pi_R}$.
\end{prop}
\begin{proof}
See \cite[Corollary 10]{lenstra-complete-intersections}.
\end{proof}
\begin{rem} \label{rem:complete_intersection_congruence_ideals}
Since $\mathrm{Fitt}_{\mathcal{O}} ( \Phi_{\pi_R} ) = \left( \lambda^{\mathrm{length}_{\mathcal{O}} (\Phi_{\pi_R})} \right)$, $\mathrm{ord}_{\lambda} (\eta_{\pi_R})$ can be written in terms of $\mathrm{length}_{\mathcal{O}} (\Phi_{\pi_R})$.
\end{rem}
\subsection{Local deformation conditions} \label{subsec:local_deformation_conditions}
For a prime $\ell$, let $\mathcal{R}(\ell)$ be the local \emph{versal} deformation ring of $\overline{\rho}\vert_{G_{\mathbb{Q}_\ell}}$ \emph{relatively representing} the deformation functor $\mathfrak{D}(\ell)$ corresponding to the local Selmer condition $L_{\ell,1} \subseteq \mathrm{H}^1(\mathbb{Q}_\ell, \mathrm{ad}^0(\overline{\rho}))$. The existence of versal deformation rings follows from \cite[Theorem 2.11]{schlessinger-criterion} even without the $\mathrm{End}_{\mathbb{F}} (\overline{\rho}\vert_{G_{\mathbb{Q}_\ell}}) = \mathbb{F}$ condition, and they are determined only up to non-canonical isomorphisms.
Thanks to \cite[Theorem 1.2.(i)]{bockle-presentations}, 
$\mathcal{R}(\ell)$ admits presentation
$$\mathcal{R}(\ell) \simeq \mathcal{O} \llbracket x_1, \cdots, x_r \rrbracket / \mathfrak{a}_\ell $$
where $r = \mathrm{dim}_{\mathbb{F}} L_{\ell,1}$. 
It is known that the number of generators of $\mathfrak{a}_\ell$ is bounded by the dimension of $\mathrm{H}^2(\mathbb{Q}_\ell, \mathrm{ad}^0(\overline{\rho}))$.
Denote by $\mathrm{gen}(\mathfrak{a}_\ell)$ the (minimal) number of generators of $\mathfrak{a}_\ell$.

We quickly review the useful local deformation problems (cf.~\cite[P1-P7]{taylor-icosahedral-2}, \cite[Definition 2.2.2]{clozel-harris-taylor}) and discuss vanishing of $\mathfrak{a}_\ell$. In \cite{clozel-harris-taylor}, the deformation problems are considered without fixing determinants, but it does not cause any problem in our setting since $p >2$.

Although we explicitly write down the local Selmer conditions only for $\mathrm{ad}^0(\overline{\rho})$ in this section,
they easily generalize to $\mathrm{ad}^0(\rho_n)$ for all $n \geq 1$.
\subsubsection{Low weight crystalline}
We first recall some of the Fontaine--Laffaille theory \cite{fontaine-laffaille}, \cite[$\S$1.1.2]{diamond-flach-guo}.
Let $\mathcal{MF}$ be the category of filtered $\varphi$-modules whose objects are
finitely generated $\mathcal{O}$-modules $M$ equipped with
\begin{itemize}
\item 
a decreasing filtration such that $\mathrm{Fil}^a M = M$ and $\mathrm{Fil}^b M = 0$ for some $a, b \in \mathbb{Z}$, and for each $i \in \mathbb{Z}$, $\mathrm{Fil}^i M$ is a direct summand of $M$;
\item 
$\mathcal{O}$-linear maps $\varphi^i : \mathrm{Fil}^i M \to M$ for $i \in \mathbb{Z}$ satisfying $\varphi^i \vert_{  \mathrm{Fil}^{i+1} M }  = p \cdot \varphi^{i+1}$ 
and $M = \sum_i \mathrm{Im}(\varphi^i)$.
\end{itemize}
It is known that $\mathcal{MF}$ is an abelian category.
Denote by $\mathcal{MF}^a$ the full subcategory of  $\mathcal{MF}$ consisting of objects $M$ satisfying 
$\mathrm{Fil}^a M = M$ and $\mathrm{Fil}^{a+p} M = 0$ and having no non-trivial quotients $M'$ with $\mathrm{Fil}^{a+p} M' = M'$.
Denote by $\mathcal{MF}^a_{\mathrm{tor}}$ the full subcategory of $\mathcal{MF}^a$ consisting of objects of finite length.
Also,  $\mathcal{MF}^a$ and $\mathcal{MF}^a_{\mathrm{tor}}$ are abelian categories and stable under taking subobjects, quotients, direct products, and extensions in  $\mathcal{MF}$.
Fontaine and Laffaille constructed the fully faithful contravariant functor from 
$\mathcal{MF}^0_{\mathrm{tor}}$ to the category of finite $\mathcal{O}[G_{\mathbb{Q}_p}]$-modules, which is called the Fontaine--Laffaille functor.

Let $R \in \mathcal{C}_\mathcal{O}$ with maximal ideal $\mathfrak{m}_R$ and $\rho\vert_{G_{\mathbb{Q}_p}} : G_{\mathbb{Q}_p} \to \mathrm{Aut}_{R}(M) \simeq \mathrm{GL}_2(R) $ be a deformation of $\overline{\rho}$, so $M$ is uniquely determined up to isomorphisms.
\begin{defn}[Low weight crystalline deformation] \label{defn:crys-low-wt}
A deformation $\rho\vert_{G_{\mathbb{Q}_p}} : G_{\mathbb{Q}_p} \to \mathrm{Aut}_{R}(M) \simeq \mathrm{GL}_2(R)$ is a \textbf{low weight crystalline deformation of $\overline{\rho}\vert_{G_{\mathbb{Q}_p}}$} if for every $n \geq 1$,  $M/\mathfrak{m}^n_R M$ lies in the image of the Fontaine--Laffaille functor on $\mathcal{MF}^0_{\mathrm{tor}}$.
\end{defn}
Let $\mathrm{H}^1_{f}(\mathbb{Q}_p, \mathrm{ad}^0(\overline{\rho})) := L_{p,1} \subseteq \mathrm{H}^1(\mathbb{Q}_p, \mathrm{ad}^0(\overline{\rho}))$ be the local condition at $p$ corresponding to the low weight crystalline deformations via isomorphism (\ref{eqn:isom-tangent}). We omit the precise definition of $L_{p,1}$. See \cite[$\S$2.1]{diamond-flach-guo}, \cite[$\S$2.4.1]{clozel-harris-taylor} for example.
Following \cite[Corollary 2.3]{diamond-flach-guo}, we have
$$\mathrm{dim}_{\mathbb{F}} L_{p,1} = \mathrm{H}^0(\mathbb{Q}_p, \mathrm{ad}^0(\overline{\rho})) +  1.$$
Furthermore, $\mathfrak{a}_p = 0$ by \cite[Lemma 2.4.1]{clozel-harris-taylor} (``liftable").

\subsubsection{Unramified}
Let $\ell \neq p$ and assume that $\overline{\rho}\vert_{G_{\mathbb{Q}_\ell}}$ is unramified; thus, $p \nmid \#\overline{\rho}(I_\ell)$.
Let $L_\ell$ be the local condition at $\ell$ corresponding to the deformation functor parametrizing all unramified deformations of $\overline{\rho}\vert_{G_{\mathbb{Q}_\ell}}$.
Then 
$$\mathrm{H}^1_{\mathrm{ur}}(G_{\mathbb{Q}_\ell}, \mathrm{ad}^0(\overline{\rho}) ) := L_{\ell,1} = \mathrm{H}^1(G_{\mathbb{Q}_\ell}/I_\ell, \mathrm{ad}^0(\overline{\rho})^{I_\ell}).$$
By \cite[E1]{taylor-icosahedral-2}, $\mathrm{H}^2(G_{\mathbb{Q}_\ell}/(I_\ell \cap \mathrm{ker}(\overline{\rho})), \mathrm{ad}^0(\overline{\rho})  ) = 0$.
It follows from \cite[Theorem 1.2.(iv)]{bockle-presentations} that $\mathfrak{a}_\ell = 0$.

\begin{lem} \label{lem:local_unramified}
Let $\ell$ be any prime (including $p$).
Let $M$ be a discrete $\mathcal{O}$-module endowed with continuous $G_{\mathbb{Q}_\ell}$-action.
For all $n \geq 1$, 
$$\# \mathrm{H}^1(G_{\mathbb{Q}_\ell}/I_\ell, M[\lambda^n]^{I_\ell}) = \#  \mathrm{H}^0(\mathbb{Q}_\ell, M[\lambda^n]) .$$
Taking the direct limit, we have
$$\mathrm{cork}_{\mathcal{O}} \ \mathrm{H}^1(G_{\mathbb{Q}_\ell}/I_\ell, M^{I_\ell})  = \mathrm{cork}_{\mathcal{O}}  \ \mathrm{H}^0(\mathbb{Q}_\ell, M).$$
\end{lem}
\begin{proof}
Comparing the kernel and the cokernel of 
$\mathrm{Frob}_\ell -1 : M[\lambda^n]^{I_\ell} \to  M[\lambda^n]^{I_\ell}$, we obtain the conclusion.
\end{proof}

\subsubsection{Unrestricted}
Let $\ell$ be a prime not dividing $p$.
It is easy to see that
$L_{\ell,1} = \mathrm{H}^1(\mathbb{Q}_\ell, \mathrm{ad}^0(\overline{\rho}))$
corresponds to the unrestricted deformations of $\overline{\rho}\vert_{G_{\mathbb{Q}_\ell}}$.
By the computation in \cite[Example 5.1.(i)]{bockle-presentations}, we have 
$$\mathrm{dim}_{\mathbb{F}}L_{\ell,1} - \mathrm{dim}_{\mathbb{F}} \mathrm{H}^1(\mathbb{Q}_\ell, \mathrm{ad}^0(\overline{\rho})) - \mathrm{gen}(\mathfrak{a}_\ell) \geq 0.$$
\subsubsection{Minimally ramified}
Let $\ell$ be a prime dividing $N(\overline{\rho})$.
Following \cite[$\S$3.1]{diamond-flach-guo} and \cite[$\S$3.2]{lundell-level-lowering}, we have the following definition.
\begin{defn} \label{defn:minimally_ramified}
A deformation $\rho\vert_{G_{\mathbb{Q}_\ell}}$ of $\overline{\rho}\vert_{G_{\mathbb{Q}_\ell}}$ to $R \in \mathrm{CNL}_{\mathcal{O}}$ is \textbf{minimally ramified} if it satisfies:
\begin{enumerate}
\item If  $p \nmid \#\overline{\rho} (I_\ell)$, then $\# \rho (I_\ell) = \# \overline{\rho}(I_\ell)$. 
\item If  $p \mid \#\overline{\rho} (I_\ell)$, then  $\rho\vert_{I_\ell}$ has unramified quotient of rank one.
\end{enumerate}
\end{defn}
The corresponding Selmer local condition at $\ell$ is denoted by $\mathrm{H}^1_{\mathrm{min}}(\mathbb{Q}_\ell, \mathrm{ad}^0(\overline{\rho})) := L_{\ell, 1}$.

Suppose that $p \nmid \#\overline{\rho} (I_\ell)$.
Then the minimally ramified liftings coincide with unramified liftings by \cite[Lemma 2.4.22]{clozel-harris-taylor}.
By applying Lemma \ref{lem:local_unramified}, $\mathrm{dim}_{\mathbb{F}} L_{\ell,1} = \mathrm{dim}_{\mathbb{F}} \mathrm{H}^0(\mathbb{Q}_\ell, \mathrm{ad}^0(\overline{\rho}))$.
By applying \cite[E1]{taylor-icosahedral-2} again, we have $\mathfrak{a}_\ell = 0$.

Suppose that $p \mid \#\overline{\rho} (I_\ell)$.
Then the minimally ramified liftings coincide with\cite[E3]{taylor-icosahedral-2}. The relevant definition and computation are given in the $\ell$-new deformation case.
\subsubsection{New} \label{subsubsec:new}
We closely follow \cite[E3]{taylor-icosahedral-2}.
Let $\ell$ be a prime not equal to $p$.
\begin{assu}
We assume one of the following conditions:
\begin{enumerate}
\item $\ell \not\equiv 1 \pmod{p}$, or 
\item $p \mid \#\overline{\rho} (G_{\mathbb{Q}_\ell})$
\end{enumerate}
\end{assu}
\begin{defn} \label{defn:new_deformations}
\begin{enumerate}
\item The first case corresponds to the type \textbf{P} in \cite[$\S$2]{diamond_extension_wiles} and 
its deformation following \cite[E3]{taylor-icosahedral-2} is called a \textbf{$\ell$-new deformation}.
\item The second case corresponds to the type \textbf{S} in \cite[$\S$2]{diamond_extension_wiles}
and its deformation following \cite[E3]{taylor-icosahedral-2} is called a \textbf{minimally ramified deformation} (of type (2) in Definition \ref{defn:minimally_ramified}).
\end{enumerate}
\end{defn}
In both cases, following \cite[$\S$2]{diamond_extension_wiles} and \cite[$\S$4.2.5]{hida-geometric-modular-forms}, there exists a choice of basis of $\mathbb{F}^2$ 
such that 
$\overline{\rho}\vert_{G_{\mathbb{Q}_\ell}} \sim \begin{pmatrix}
\mathbf{1} & \overline{\xi} \\ 0 & \overline{\chi}^{-1}_{\mathrm{cyc}}
\end{pmatrix} \otimes \overline{\chi}^{1 - k/2}_{\mathrm{cyc}}$
 where $\overline{\xi} \in Z^1(\mathbb{Q}_\ell, \mathbb{F}(1-k/2))$.
Since the universal deformation rings of equivalent representations up to a character twist are canonically isomorphic \cite[Proposition 1]{mazur-deforming}, we may assume
$\overline{\rho}\vert_{G_{\mathbb{Q}_\ell}} \sim \begin{pmatrix}
\mathbf{1} & \overline{\xi} \\ 0 & \overline{\chi}^{-1}_{\mathrm{cyc}}
\end{pmatrix}$
without loss of generality.

We consider the collection of the deformations $\rho \vert_{G_{\mathbb{Q}_\ell}}$ of the following form
$\begin{pmatrix}
\mathbf{1} & \xi \\ 0 & \chi^{-1}_{\mathrm{cyc}}
\end{pmatrix}$.
In order to define the local condition $L_{\ell,1} \subseteq \mathrm{H}^1(\mathbb{Q}_\ell, \mathrm{ad}^0(\overline{\rho}))$ corresponding to the deformations,
we first explicitly describe how $\mathrm{ad}^0( \rho  ) \vert_{G_{\mathbb{Q}_\ell}}$ looks like.
A straightforward computation shows that 
\begin{equation}\label{simple-computation}
\mathrm{ad}^0( \rho  ) \vert_{G_{\mathbb{Q}_\ell}} \sim 
\begin{pmatrix}
\chi_{\mathrm{cyc}} & - 2\xi \chi_{\mathrm{cyc}} & -\xi^2 \chi_{\mathrm{cyc}} \\
0 & \mathbf{1}  & \xi  \\
0 & 0 &  \chi^{-1}_{\mathrm{cyc}}
\end{pmatrix}
 . 
 \end{equation}
Let $M = \mathrm{ad}^0(\rho) \vert_{G_{\mathbb{Q}_\ell}} \otimes E/\mathcal{O}$ be the discrete $G_{\mathbb{Q}_\ell}$-module corresponding to the matrix form above and consider $G_{\mathbb{Q}_\ell}$-stable filtration
$$0 \subsetneq M_2 \subsetneq M_1 \subsetneq M_0 = M$$
induced by (\ref{simple-computation}).

Let $\rho_{n} \vert_{G_{\mathbb{Q}_\ell}}$ be the mod $\lambda^n$ reduction of $\rho \vert_{G_{\mathbb{Q}_\ell}}$ and
$M[\lambda^n]$ be the $\lambda^n$-torsion of $M$. Then we have
$\mathrm{ad}^0(\rho_{n} \vert_{G_{\mathbb{Q}_\ell}}) \otimes E/\mathcal{O} \simeq \mathrm{ad}^0(\rho \vert_{G_{\mathbb{Q}_\ell}}) \otimes \lambda^{-n}\mathcal{O}/\mathcal{O}$.
Note that the ramification of $\rho_{n} \vert_{G_{\mathbb{Q}_\ell}}$ is completely controlled by the 1-cocycle $\xi \pmod{\lambda^n}$.
More explicitly, considering the equation
\begin{equation} \label{eqn:ad^0_f_explicit_matrix}
\begin{pmatrix}
\chi_{\mathrm{cyc}} & - 2\xi \chi_{\mathrm{cyc}} & -\xi^2 \chi_{\mathrm{cyc}} \\
0 & \mathbf{1}  & \xi  \\
0 & 0 &  \chi^{-1}_{\mathrm{cyc}}
\end{pmatrix} 
\begin{pmatrix}
a \\
b \\
c
\end{pmatrix}
=
\begin{pmatrix}
\chi_{\mathrm{cyc}} \cdot a  - 2\xi \chi_{\mathrm{cyc}} \cdot b  -\xi^2 \chi_{\mathrm{cyc}} \cdot c \\
 b + \xi  \cdot c \\
\chi^{-1}_{\mathrm{cyc}} \cdot c
\end{pmatrix} ,
\end{equation}
we can explicitly observe that $M_2$ is generated by ${}^{t}\begin{pmatrix}
1 &
0 &
0
\end{pmatrix}
$ and $M_1$ is generated by 
${}^{t}\begin{pmatrix}
1 &
0 &
0
\end{pmatrix}$
and 
${}^{t}\begin{pmatrix}
0 &
1 &
0
\end{pmatrix}$.
Thus, we have $\mathrm{H}^1(\mathbb{Q}_\ell, M_2) \simeq \mathrm{H}^1(\mathbb{Q}_\ell, E/\mathcal{O}(1))$.
The local Selmer condition at $\ell$ corresponding to the deformations above is defined by
$$\mathrm{H}^1_{\mathrm{new}} (\mathbb{Q}_\ell, \mathrm{ad}^0(\overline{\rho}) ) = L_{\ell, 1} := \mathrm{Im} \left(  \mathrm{H}^1(\mathbb{Q}_\ell, M_2[\lambda]) \to  \mathrm{H}^1(\mathbb{Q}_\ell, \mathrm{ad}^0(\overline{\rho})) \right) \subseteq  \mathrm{H}^1(\mathbb{Q}_\ell, \mathrm{ad}^0(\overline{\rho})) $$ 
when $\overline{\rho}$ is unramified at $\ell$.
Even when $\overline{\rho}$ is ramified at $\ell$, we  also denote it by
$ \mathrm{H}^1_{\mathrm{new}} (\mathbb{Q}_\ell, \mathrm{ad}^0(\overline{\rho}) ) =  \mathrm{H}^1_{\mathrm{min}} (\mathbb{Q}_\ell, \mathrm{ad}^0(\overline{\rho}) )$ for convenience.

\begin{lem}
Suppose that $L_{\ell, 1}$ is the local Selmer condition corresponding to $\ell$-new or minimally ramified deformations as in Definition \ref{defn:new_deformations}.
Then
\begin{enumerate}
\item $\mathrm{dim}_{\mathbb{F}} L_{\ell, 1} = \mathrm{dim}_{\mathbb{F}} \mathrm{H}^1(\mathbb{Q}_\ell , \mathrm{ad}^0(\overline{\rho}))$, and
\item  the versal deformation ring of $\overline{\rho}\vert_{G_{\mathbb{Q}_\ell}}$ corresponding to $L_{\ell, 1}$ is smooth over $\mathcal{O}$, i.e. unobstructed.
\end{enumerate}
\end{lem}
\begin{proof}
See \cite[E3]{taylor-icosahedral-2}.
\end{proof}
\subsection{The deformation ring and the adjoint Selmer group}
Putting all the local deformation conditions discussed together, we define the following deformation ring.

Let $\Sigma^+$ and $\Sigma^-$ be two finite sets of primes $\ell (\neq p)$ such that $\Sigma^+ \cap \Sigma^- = \emptyset$. 
\begin{defn}[$\Sigma^+$-ramified $\Sigma^-$-new deformation rings]
Let $\mathfrak{D}^{\Sigma^-}_\Sigma$ be the deformation functor satisfying the following local deformation conditions:
\begin{enumerate}
\item The local deformation at $p$ is a low weight crystalline deformation (Definition \ref{defn:crys-low-wt});
\item At $\ell \not\in \Sigma^+ \cup \Sigma^-$, the local deformation at $\ell$ is minimally ramified;
\item At $\ell \in \Sigma^+$, the local deformation at $\ell$ is unrestricted;
\item At $\ell \in \Sigma^-$, the local deformation at $\ell$ is new if $\overline{\rho}$ is unramified at $\ell$;
\item At $\ell \in \Sigma^-$, the local deformation at $\ell$ is minimally ramified if $\overline{\rho}$ is ramified at $\ell$.
\end{enumerate}
The deformation ring representing the functor $\mathfrak{D}^{\Sigma^-}_{\Sigma^+}$ is called the \textbf{$\Sigma^+$-ramified $\Sigma^-$-new  deformation ring} and denoted by
$\mathcal{R}^{\Sigma^-}_{\Sigma^+}$, and denote by $\rho^{\Sigma^-}_{\Sigma^+}$ the corresponding representation. If $\Sigma^- = \emptyset$, we write
$\mathcal{R}_{\Sigma^+} = \mathcal{R}^{\emptyset}_{\Sigma^+}$ and $\rho_{\Sigma^+} = \rho^{\emptyset}_{\Sigma^+}$.
\end{defn}
\begin{rem}
The representability of $\mathfrak{D}^{\Sigma^-}_{\Sigma^+}$ is ensured by the argument described in $\S$\ref{subsec:local-global-deformation}.
See also \cite{khare-R=T, khare-ramakrishna} (depending on \cite{ribet-takahashi}), and \cite{yu-higher-khare-ramakrishna} for another description.
\end{rem}

Let $f$ be a newform such that $\rho_f$ is a $\Sigma^+$-ramified $\Sigma^-$-new deformation of $\overline{\rho}$.
Then we define the adjoint Selmer group of $f$ corresponding to the deformation problem following (\ref{eqn:tangent-selmer}) as follows. For notational convenience, write $M =\mathrm{ad}^0(f) \otimes E/\mathcal{O}$.
\begin{defn} \label{defn:adjoint_selmer}
The \textbf{$\Sigma^+$-imprimitive $\Sigma^-$-new adjoint Selmer group $\mathrm{Sel}^{\Sigma^-}_{\Sigma^+}(\mathbb{Q}, M)$ of $f$}
 is defined by the kernel of the natural restriction map
$$\phi^{\Sigma^-}_{\Sigma^+}  : \mathrm{H}^1(\mathbb{Q}_{S \cup \Sigma^+  \cup \Sigma^-}/\mathbb{Q}, M) \to \dfrac{\mathrm{H}^1(\mathbb{Q}_p, M)}{\mathrm{H}^1_{f}(\mathbb{Q}_p, M)} \oplus \bigoplus_{\ell \in S \setminus \Sigma^+} \dfrac{\mathrm{H}^1(\mathbb{Q}_\ell, M)}{\mathrm{H}^1_{\mathrm{min}}(\mathbb{Q}_\ell, M)} \oplus \bigoplus_{\ell \in \Sigma^- } \dfrac{\mathrm{H}^1(\mathbb{Q}_\ell, M)}{\mathrm{H}^1_{\mathrm{new}}(\mathbb{Q}_\ell, M)} .$$
\end{defn}

\subsection{A presentation of the deformation ring}
We quickly summarize \cite[$\S$5]{bockle-presentations}.
It turns out that many naturally defined local deformation conditions satisfy the following properties.
\begin{enumerate}
\item 
If $\ell$ does not divide $p$, then 
$$\mathrm{dim}_{\mathbb{F}} L_{\ell,1} - 
\mathrm{dim}_{\mathbb{F}} \mathrm{H}^0(\mathbb{Q}_\ell, \mathrm{ad}^0(\overline{\rho})) -
 \mathrm{gen}(\mathfrak{a}_\ell) \geq 0 .$$
\item If we impose a suitable semistable assumption on deformations at $p$, then 
$$\sum_{v \vert p}
\left( 
\mathrm{dim}_{\mathbb{F}} L_{p,1} - 
\mathrm{dim}_{\mathbb{F}} \mathrm{H}^0(\mathbb{Q}_p, \mathrm{ad}^0(\overline{\rho})) -
 \mathrm{gen}(\mathfrak{a}_p) 
\right) \geq 0 .$$
\end{enumerate}
Note that the condition at the infinite place is automatic since $\overline{\rho}$ is odd.
If we further assume that $\mathrm{gen}(\mathfrak{a}_p) \leq 1$, then the global deformation ring with these local constraints is a complete intersection. The precise statement is as follows.
\begin{thm}[B\"{o}ckle]
Let $\overline{\rho} : G_\mathbb{Q} \to \mathrm{GL}_2(\overline{\mathbb{F}}_p)$ be an odd, continuous, and absolutely irreducible Galois representation.
We assume the following conditions.
\begin{enumerate}
\item If $\ell$ does not divide $p$, then 
$$\mathrm{dim}_{\mathbb{F}} L_{\ell, 1} - 
\mathrm{dim}_{\mathbb{F}} \mathrm{H}^0(\mathbb{Q}_\ell, \mathrm{ad}^0(\overline{\rho})) -
 \mathrm{gen}(\mathfrak{a}_\ell) \geq 0 .$$
\item If $\ell = p$, then $\overline{\rho}\vert_{G_{\mathbb{Q}_p}}$ is low weight crystalline and
$$
\mathrm{dim}_{\mathbb{F}} L_{p,1} - 
\mathrm{dim}_{\mathbb{F}} \mathrm{H}^0(\mathbb{Q}_p, \mathrm{ad}^0(\overline{\rho})) -
 \mathrm{gen}(\mathfrak{a}_p) 
\geq 0 .$$
\item $\overline{\rho}$ satisfies Assumption \ref{assu:working_assumptions}.(TW).
\end{enumerate}
Then the corresponding deformation ring $\mathcal{R}$ is a complete intersection over $\mathcal{O}$, i.e.
$$\mathcal{R} \simeq \mathcal{O} \llbracket X_1, \cdots, X_n \rrbracket / (f_1, \cdots f_n)$$
for suitable $f_i \in \mathcal{O} \llbracket X_1, \cdots, X_n \rrbracket$.
\end{thm}
\begin{proof}
See \cite[Corollary 4.3 and Theorem 5.8]{bockle-presentations}.
Note that
$\mathrm{H}^0(G_{\mathbb{Q}, S}, \mathrm{ad}^0(\overline{\rho})^\vee) = 0$ 
where $M^\vee = \mathrm{Hom}_{\mathbb{F}}(M,\mathbb{F})(1)$
since
$\overline{\rho}$ satisfies Assumption \ref{assu:working_assumptions}.(TW).
\end{proof}

\begin{cor} \label{cor:deformation_ring_complete_intersection}
The deformation ring $\mathcal{R}^{\Sigma^-}_{\Sigma^+}$ is a complete intersection; thus, we have
$$\mathcal{R}^{\Sigma^-}_{\Sigma^+} \simeq \mathcal{O}\llbracket X_1, \cdots, X_n \rrbracket / (f_1, \cdots, f_n)$$
\end{cor}
\begin{proof}
Considering all the local deformation conditions in $\S$\ref{subsec:local_deformation_conditions}, the conclusion immediately follows.
\end{proof}

\subsection{A refined $R = \mathbb{T}$ theorem} \label{subsec:refined_R=T}
The goal of this section is to prove the following theorem.

\begin{thm} \label{thm:refined_R=T}
Keep Assumption \ref{assu:working_assumptions}.
Let $N \geq 1$ be an integer such that $N(\overline{\rho})$ divides $N$.
Let $\Sigma$ and $\Sigma^+$ be two finite sets of primes not equal to $p$ such that 
$P_{\overline{\rho}} \subseteq \Sigma$. Write $\Sigma^- = \Sigma \setminus \Sigma^+$.
There exists an isomorphism
$$\pi^{\Sigma^-}_{\Sigma^+} : \mathcal{R}^{\Sigma^-}_{\Sigma^+} \simeq \mathbb{T}(N^+_{\overline{\rho},\Sigma^+} N^-)^{N^-}_{ \mathfrak{m}_{\Sigma^+} }$$
of reduced finite flat complete intersections over $\mathcal{O}$.
\end{thm}
Recall that
$\mathbb{T}(N^+_{\overline{\rho},\Sigma^+} N^-)_{ \mathfrak{m}_{\Sigma^+} }$
is isomorphic to the image of the abstract $S$-anemic Hecke algebra $\mathbb{T}^S$ in 
$\mathrm{End}_{\mathcal{O}} ( \mathrm{H}^1_{\mathrm{\acute{e}t},c}(Y_{1}(N^+_{\overline{\rho},\Sigma^+} N^-, \mathcal{F}^k_p)[\langle-\rangle - \mathbf{1}]_{ \mathfrak{m}_{\Sigma^+} } )$ with $\Sigma \subseteq S$
and
$\mathbb{T}(N^+_{\overline{\rho},\Sigma^+} N^-)^{N^-}_{ \mathfrak{m}_{\Sigma^+} }$ is the $N^-$-new quotient of 
$\mathbb{T}(N^+_{\overline{\rho},\Sigma^+} N^-)_{ \mathfrak{m}_{\Sigma^+} }$.

Before giving a proof, we recall two lemmas.
\begin{lem} \label{lem:finite_flat_complet_intersection}
Suppose that $R \simeq \mathcal{O} \llbracket X_1, \cdots , X_n \rrbracket / (f_1, \cdots, f_m)$ is a finite $\mathcal{O}$-algebra with $m \leq n$. Then $m = n$ and $R$ is a finite flat complete intersection.
\end{lem}
\begin{proof}
See \cite[Lemma 2.3]{lundell-level-lowering}.
\end{proof}
\begin{lem} \label{lem:reduced_quotient}
Suppose that $R$ is a finite flat reduced $\mathcal{O}$-algebra.
Let $R/\mathfrak{a}$ be a quotient of $R$ which is a finite flat $\mathcal{O}$-algebra.
Then $R/\mathfrak{a}$ is reduced. 
\end{lem}
\begin{proof}
See \cite[Lemma 2.4]{lundell-level-lowering}.
\end{proof}
\begin{proof}[Proof of Theorem \ref{thm:refined_R=T}]
By Theorem \ref{thm:diamond-flach-guo-dimitrov}, we have
an isomorphism 
$$ \pi_\Sigma : \mathcal{R}_{\Sigma} \simeq \mathcal{T}_{\Sigma}$$
of reduced finite flat complete intersections over $\mathcal{O}$.
Note that the reduced and finite flat properties come from $\mathcal{T}_{\Sigma}$ (cf.~\cite[$\S$1.2 and Theorem 6.2]{dimitrov-ihara}).
Since $\mathcal{R}^{\Sigma^-}_{\Sigma^+}$ is a quotient of $\mathcal{R}_{\Sigma}$, $\mathcal{R}^{\Sigma^-}_{\Sigma^+}$ is finite over $\mathcal{O}$. 
By Corollary \ref{cor:deformation_ring_complete_intersection}, $\mathcal{R}^{\Sigma^-}_{\Sigma^+}$ is a complete intersection.
Furthermore, the flatness and the reducedness of $\mathcal{R}^{\Sigma^-}_{\Sigma^+}$ follows from Lemma \ref{lem:finite_flat_complet_intersection} and Lemma \ref{lem:reduced_quotient}, respectively.
We explicitly construct the map
$$\pi^{\Sigma^-}_{\Sigma^+} : \mathcal{R}^{\Sigma^-}_{\Sigma^+} \to \mathbb{T}(N^+_{\overline{\rho},\Sigma^+} N^-)^{N^-}_{ \mathfrak{m}_{\Sigma^+} }$$
satisfying the commutative diagram
\begin{equation} \label{eqn:r=t_diagram}
\begin{gathered}
\xymatrix{
\mathcal{R}_{\Sigma} \ar[rr]^-{\pi_{\Sigma}}_{\textrm{Thm. \ref{thm:diamond-flach-guo-dimitrov}}} \ar@{->>}[dd]  & & \mathcal{T}_{\Sigma} \ar[r]^-{\simeq}_-{\textrm{Lem. \ref{lem:sigma-level-identification}}} & \mathbb{T}(N_{\overline{\rho},\Sigma})_{ \mathfrak{m}_{\Sigma} } \ar@{->>}[d] \\
&  &  & \mathbb{T}(N^+_{\overline{\rho},\Sigma^+} N^-)_{ \mathfrak{m}_{\Sigma^+} } \ar@{->>}[d] \\
\mathcal{R}^{\Sigma^-}_{\Sigma^+} \ar[rrr]^-{\pi^{\Sigma^-}_{\Sigma^+}  } & & & \mathbb{T}(N^+_{\overline{\rho},\Sigma^+} N^-)^{N^-}_{ \mathfrak{m}_{\Sigma^+} }
}
\end{gathered}
\end{equation}
and show that it is an isomorphism.

Suppose that there is an $\mathcal{O}$-algebra morphism $\alpha : \mathcal{R}^{\Sigma^-}_{\Sigma^+} \to \mathcal{O}'$ where $\mathcal{O}'$ is a domain of characteristic zero and the corresponding deformation is denoted by $\rho'$.
By making the following composition
\begin{equation} \label{eqn:hecke_deformation_quotient}
\xymatrix{
\mathbb{T}(N_{\overline{\rho},\Sigma})_{ \mathfrak{m}_{\Sigma} } \ar[r]^-{\pi^{-1}_\Sigma}_-{\simeq} & \mathcal{R}_{\Sigma} \ar@{->>}[r] & \mathcal{R}^{\Sigma^-}_{\Sigma^+} \ar[r]^-{\alpha} & \mathcal{O}' ,
}
\end{equation}
there exists a newform $g$ of level dividing $N_{\overline{\rho},\Sigma}$ and $\rho'$ and $\rho_{g}$ are obviously equivalent.

Let $\ell$ be a prime dividing $N^-$ not dividing $N(\overline{\rho})$.
Since $\rho' \vert_{G_{\mathbb{Q}_\ell}}$ is a $\ell$-new deformation at $\ell$, we have
$$\mathrm{tr} \ \rho ' ( \mathrm{Frob}_\ell) = \pm \ell^{\frac{k-2}{2}} (\ell +1).$$
By considering the Ramanujan--Petersson bound (at good primes), it is easy to see that $\rho'$ is ramified at all primes dividing $N^-$.
Thus, $g$ is new at all primes dividing $N^-$ and the map (\ref{eqn:hecke_deformation_quotient}) factors through $\mathbb{T}(N^+_{\overline{\rho},\Sigma^+}  N^-)^{N^-}_{ \mathfrak{m}_{\Sigma^+} }$ and we obtain a map $\beta :  \mathbb{T}(N^+_{\overline{\rho},\Sigma^+}  N^-)^{N^-}_{ \mathfrak{m}_{\Sigma^+} } \to \mathcal{O}'$.
The universality of $\mathcal{R}^{\Sigma^-}_{\Sigma^+}$ induces the surjective map
$$\pi^{\Sigma^-}_{\Sigma^+} : \mathcal{R}^{\Sigma^-}_{\Sigma^+} \to \mathbb{T}(N^+_{\overline{\rho},\Sigma^+} N^-)^{N^-}_{ \mathfrak{m}_{\Sigma^+} }$$
such that $\alpha = \beta \circ \pi^{\Sigma^-}_{\Sigma^+}$.
Since $\alpha = \beta \circ \pi^{\Sigma^-}_{\Sigma^+}$, we have the bijection between
\begin{itemize}
\item the characteristic zero minimal prime ideals of $\mathcal{R}^{\Sigma^-}_{\Sigma^+}$ and
\item the characteristic zero minimal prime ideals of $\mathbb{T}(N^+_{\overline{\rho},\Sigma^+} N^-)^{N^-}_{ \mathfrak{m}_{\Sigma^+} }$.
\end{itemize}
We now claim the injectivity of $\pi^{\Sigma^-}_{\Sigma^+}$. Due to the bijection between minimal primes above, the kernel of $\pi^{\Sigma^-}_{\Sigma^+}$ is contained in the intersection of all the characteristic zero minimal prime ideals of $\mathcal{R}^{\Sigma^-}_{\Sigma^+}$. In other words,
$$\mathrm{ker} (  \pi^{\Sigma^-}_{\Sigma^+} ) \subseteq \bigcap_{\substack{\mathrm{min} \\ \mathrm{char} \ 0 }}\wp .$$
Here, $\bigcap_{\substack{\mathrm{min} \\ \mathrm{char} \ 0 }}$ means that ideal $\wp$ runs over the set of minimal ideals of $\mathcal{R}^{\Sigma^-}_{\Sigma^+}$ of characteristic zero. 
Since $\mathcal{O} \to \mathcal{R}^{\Sigma^-}_{\Sigma^+}$ is finite flat, a uniformizer $\lambda$ of $\mathcal{O}$ maps to a non-zero divisor of $\mathcal{R}^{\Sigma^-}_{\Sigma^+}$.  Thus, any minimal prime $\wp$ does not contain $\lambda$.
Thus, we obtain the conclusion due to the reduced property of $\mathcal{R}^{\Sigma^-}_{\Sigma^+}$.
\end{proof}
\begin{rem} \label{rem:local_deformation_conditions}
In (\ref{eqn:r=t_diagram}), one may expect the existence of the deformation ring $``\mathcal{R}^{\Sigma^-\textrm{-ss}}_{\Sigma^+}"$ isomorphic to $\mathbb{T}(N^+_{\overline{\rho},\Sigma^+} N^-)$. If so, the following diagram would commute:
\[
\xymatrix{
\mathcal{R}_{\Sigma} \ar[rr]^-{\pi_{\Sigma}}_{\textrm{Thm. \ref{thm:diamond-flach-guo-dimitrov}}} \ar@{->>}[d]  & & \mathcal{T}_{\Sigma} \ar[r]^-{\simeq}_-{\textrm{Lem. \ref{lem:sigma-level-identification}}} & \mathbb{T}(N_{\overline{\rho},\Sigma})_{ \mathfrak{m}_{\Sigma} } \ar@{->>}[d] \\
``\mathcal{R}^{\Sigma^-\textrm{-ss}}_{\Sigma^+}" \ar@{->>}[d] \ar[rrr]^-{\simeq} &  &  & \mathbb{T}(N^+_{\overline{\rho},\Sigma^+} N^-)_{ \mathfrak{m}_{\Sigma^+} } \ar@{->>}[d] \\
\mathcal{R}^{\Sigma^-}_{\Sigma^+} \ar[rrr]^-{\pi^{\Sigma^-}_{\Sigma^+}  } & & & \mathbb{T}(N^+_{\overline{\rho},\Sigma^+} N^-)^{N^-}_{ \mathfrak{m}_{\Sigma^+} } .
}
\]
However, it looks difficult to impose the right local deformation condition at primes dividing $N^-$ (unless $\overline{\rho}$ is ramified at all primes dividing $N^-$) since the local deformation condition at unramified primes dividing $N^-$ should include both unramified and new deformations. This is also pointed out in \cite[$\S$9]{dummigan-higher-congruences}.
\end{rem}
\section{Relative computation of adjoint Selmer groups} \label{sec:adjoint_Selmer}
The goal of this section is to prove Theorem \ref{thm:key_thm} (the Selmer computation).
\subsection{Preliminaries on Galois cohomology}
Let $T$ be a free $\mathcal{O}$-module of rank $d$ endowed with continuous action of $G_{\mathbb{Q}}$ and $S$ a finite set of places of $\mathbb{Q}$ containing $p$, $\infty$, and the ramified primes for $T$.
In other words, we have a continuous $d$-dimensional integral Galois representation
$$\rho : G_{\mathbb{Q},S} \to \mathrm{GL}_d(\mathcal{O}) \simeq \mathrm{Aut}_{\mathcal{O}}(T) .$$
Let $A = T \otimes_{\mathcal{O}} E/\mathcal{O}$ be the associated discrete Galois module.
For a Selmer structure $\mathcal{L} = (L_\ell)_{\ell \in S}$ with $L_\ell \subseteq \mathrm{H}^1(\mathbb{Q}_\ell, A)$,
we define the \textbf{(discrete) Selmer group of $A$ with respect to $\mathcal{L}$} by 
$$\mathrm{Sel}_{\mathcal{L}} (G_{\mathbb{Q},S}, A) := \mathrm{ker} \left( 
\phi_{\mathcal{L}} : \mathrm{H}^1 (G_{\mathbb{Q},S}, A) \to \prod_{\ell \in S}  \dfrac{ \mathrm{H}^1 (\mathbb{Q}_\ell, A) }{ L_\ell}
\right)$$

The Tate local duality gives us the non-degenerate pairing 
$$\mathrm{H}^1(\mathbb{Q}_\ell, A) \times \mathrm{H}^1(\mathbb{Q}_\ell, T^*) \to E/\mathcal{O} $$
where $T^* := \mathrm{Hom}(A, E/\mathcal{O}(1))$.
For a Selmer structure $\mathcal{L}$ for $S$ and $A$,
we define the dual Selmer structure $\mathcal{L}^* = (L^{*}_\ell)_{\ell \in S}$ for $S$ and $T^*$
by $L^*_\ell := L^\perp_\ell$ under the pairing.
Then we define the \textbf{dual (compact) Selmer group of $T^*$ with respect to $\mathcal{L}^*$} by
$$\mathrm{Sel}_{\mathcal{L}^*} (G_{\mathbb{Q},S}, T^*) := \mathrm{ker} \left( \phi_{\mathcal{L}^*} : \mathrm{H}^1 (G_{\mathbb{Q},S}, T^*) \to \prod_{\ell \in S} \dfrac{ \mathrm{H}^1 (\mathbb{Q}_\ell, T^*) }{ L^{*}_\ell }
\right) .$$

The comparison between two Selmer groups sometimes reduces to the comparison of local conditions under the surjectivity of the global-to-local map defining the smaller Selmer group.
\begin{prop} \label{prop:comparison_selmer}
Let $\mathcal{L}$ and $\mathcal{N}$ be two Selmer structures for $S$ and $A$.
If $L_\ell \subseteq N_\ell$ for all $\ell \in S$, then
$$\mathrm{H}^1_{\mathcal{L}} (G_{\mathbb{Q},S}, A) \subseteq \mathrm{H}^1_{\mathcal{N}} (G_{\mathbb{Q},S}, A) .$$
If we further assume that $\phi_{\mathcal{L}}$ is surjective, then we have
$$\dfrac{ \mathrm{H}^1_{\mathcal{N}} (G_{\mathbb{Q},S}, A) }{ \mathrm{H}^1_{\mathcal{L}} (G_{\mathbb{Q},S}, A) } \simeq \prod_{\ell \in S} \dfrac{ N_v}{ L_v }.$$
\end{prop}
\begin{proof}
It immediately follows from the surjectivity of $\phi_{\mathcal{L}}$.
\end{proof}

The following proposition is the direct limit version of the formula of Greenberg--Wiles \cite[Proposition 1.6]{wiles}, which is an application of the Poitou--Tate exact sequence with Selmer structures. See also \cite[Lemma 2.6]{lundell-level-lowering}.
\begin{prop} \label{prop:greenberg-wiles}
Let $\mathcal{L}$ be a Selmer structure for $S$ and $A$.
Then  
$\mathrm{Sel}_{\mathcal{L}} (G_{\mathbb{Q},S}, A)$ is cofinitely generated and 
$\mathrm{Sel}_{\mathcal{L}^*} (G_{\mathbb{Q},s}, T^*)$ is finitely generated
as $\mathcal{O}$-modules. Moreover, we have an equality
\begin{align*}
&
\mathrm{cork}_{\mathcal{O}} \ \mathrm{H}^1_{\mathcal{L}} (G_{\mathbb{Q},S}, A)
- \mathrm{rk}_{\mathcal{O}} \ \mathrm{H}^1_{\mathcal{L}^*} (G_{\mathbb{Q},S}, T^*) \\
= \ \ &
\mathrm{cork}_{\mathcal{O}} \ \mathrm{H}^0 (G_{\mathbb{Q},S}, A)
- \mathrm{rk}_{\mathcal{O}} \ \mathrm{H}^0 (G_{\mathbb{Q},S}, T^*)
+ 
\sum_{\ell \in S} \left(  \mathrm{cork}_{\mathcal{O}} \ L_\ell
- \mathrm{cork}_{\mathcal{O}} \ \mathrm{H}^0 (\mathbb{Q}_\ell , A)
 \right) .
\end{align*}
\end{prop}

In order to have the surjectivity of the global-to-local map defining Selmer groups, the local conditions should be ``well-balanced" as follows. This is \cite[Proposition 2.7]{lundell-level-lowering} and we include the proof for the completeness.
\begin{prop} \label{prop:surjectivity-of-Selmer}
Assume that $A[\lambda]$ and $T^*/ \lambda T^*$ are irreducible as $G_{\mathbb{Q},S}$-modules.
If
$\mathrm{Sel}_{\mathcal{L}} (G_{\mathbb{Q},S}, A)$ is finite, and
$$\sum_{\ell \in S}   \mathrm{cork}_{\mathcal{O}} \ L_\ell
= \sum_{\ell \in S}  \mathrm{cork}_{\mathcal{O}} \ \mathrm{H}^0 (\mathbb{Q}_\ell , A),$$
then the global-to-local map $\phi_{\mathcal{L}}$ is surjective.
\end{prop}
\begin{proof}
By Proposition \ref{prop:greenberg-wiles}, 
the finiteness of $\mathrm{H}^1_{\mathcal{L}} (G_{\mathbb{Q},S}, A)$ implies 
$\mathrm{rk}_{\mathcal{O}} \ \mathrm{H}^1_{\mathcal{L}^*} (G_{\mathbb{Q},S}, T^*) = 0$, so 
$\mathrm{H}^1_{\mathcal{L}^*} (G_{\mathbb{Q},S}, T^*)$ is also finite.
Thus, $\mathrm{H}^1_{\mathcal{L}^*} (G_{\mathbb{Q},S}, T^*)$ is contained in the $\mathcal{O}$-torsion of $\mathrm{H}^1 (G_{\mathbb{Q},S}, T^*)$. Due to the irreducibility assumption, $\mathrm{H}^1 (G_{\mathbb{Q},S}, T^*)$ is torsion free as an $\mathcal{O}$-module following the argument in \cite[$\S$2.2]{greenberg-surjectivity}. Thus, $\mathrm{H}^1_{\mathcal{L}^*} (G_{\mathbb{Q},S}, T^*) = 0$. By \cite[Proposition 3.1.1]{greenberg-surjectivity}, $\phi_{\mathcal{L}}$ is surjective.
\end{proof}

\subsection{Local computation}
We quickly recall some materials in $\S$\ref{subsubsec:new}.
Let $f$ be a newform such that $\rho_f \simeq \rho$ such that $\rho \vert_{G_{\mathbb{Q}_\ell}}$ is an $\ell$-new deformation of $\overline{\rho} \vert_{G_{\mathbb{Q}_\ell}}$.
More explicitly,
$\overline{\rho} \vert_{G_{\mathbb{Q}_\ell}}$ is unramified but $\rho \vert_{G_{\mathbb{Q}_\ell}}$ ramified and is twist-equivalent to
$$\begin{pmatrix}
\mathbf{1} & \xi \\
0 & \chi^{-1}_{\mathrm{cyc}}
\end{pmatrix}$$
where $\xi \in Z^1(\mathbb{Q}_\ell, \mathcal{O}(1-k/2))$.
Then $\mathrm{ad}^0( \rho  ) \vert_{G_{\mathbb{Q}_\ell}}$ is equivalent to
$$
\begin{pmatrix}
\chi_{\mathrm{cyc}} & - 2\xi \chi_{\mathrm{cyc}} & -\xi^2 \chi_{\mathrm{cyc}} \\
0 & \mathbf{1}  & \xi  \\
0 & 0 &  \chi^{-1}_{\mathrm{cyc}}
\end{pmatrix} .
$$
Let $M = \mathrm{ad}^0(\rho) \vert_{G_{\mathbb{Q}_\ell}} \otimes E/\mathcal{O}$ with $G_{\mathbb{Q}_\ell}$-stable filtration
$$0 \subsetneq M_2 \subsetneq M_1 \subsetneq M_0 .$$
More explicitly, recall the equation (\ref{eqn:ad^0_f_explicit_matrix})
$$\begin{pmatrix}
\chi_{\mathrm{cyc}} & - 2\xi \chi_{\mathrm{cyc}} & -\xi^2 \chi_{\mathrm{cyc}} \\
0 & \mathbf{1}  & \xi  \\
0 & 0 &  \chi^{-1}_{\mathrm{cyc}}
\end{pmatrix} 
\begin{pmatrix}
a \\
b \\
c
\end{pmatrix}
=
\begin{pmatrix}
\chi_{\mathrm{cyc}} \cdot a  - 2\xi \chi_{\mathrm{cyc}} \cdot b  -\xi^2 \chi_{\mathrm{cyc}} \cdot c \\
 b + \xi  \cdot c \\
\chi^{-1}_{\mathrm{cyc}} \cdot c
\end{pmatrix} .
$$
Since $M_2$ is generated by ${}^{t}\begin{pmatrix}
1 &
0 &
0
\end{pmatrix}
$ and $M_1$ is generated by 
${}^{t}\begin{pmatrix}
1 &
0 &
0
\end{pmatrix}$
and 
${}^{t}\begin{pmatrix}
0 &
1 &
0
\end{pmatrix}$, we have the following computation
\begin{align} \label{eqn:counting_H^0}
\begin{split}
\mathrm{H}^0(\mathbb{Q}_\ell, M_2[\lambda^n]) & \simeq \lambda^{-s}\mathcal{O}/\mathcal{O} \\
\mathrm{H}^0(\mathbb{Q}_\ell, M_0[\lambda^n]) & \simeq \lambda^{-s}\mathcal{O}/\mathcal{O} \oplus \lambda^{-t}\mathcal{O}/\mathcal{O} \oplus   \lambda^{-\mathrm{min}(s,t)}\mathcal{O}/\mathcal{O} \\
\mathrm{H}^0(\mathbb{Q}_\ell, M_0[\lambda^n]/M_2[\lambda^n]) & \simeq \lambda^{-n}\mathcal{O}/\mathcal{O} \oplus \lambda^{-\mathrm{min}(s,t)}\mathcal{O}/\mathcal{O}
\end{split}
\end{align}
for $n \gg 0$
where 
\begin{itemize}
\item
$s$ is the largest integer such that $\chi_{\mathrm{cyc}}(g) c \equiv c \pmod{\lambda^s}$, i.e. $s = \mathrm{ord}_{\lambda}(\ell - 1)$, and
\item 
$t$ is the largest integer such that $\rho_t \vert_{G_{\mathbb{Q}_\ell}}$ is semi-simple, i.e. $t = t_f(\ell)$ (Definition \ref{defn:tamagawa_exponents}).
\end{itemize}
Since the cohomological dimension of $G_{\mathbb{Q}_\ell}/I_\ell$ is one, the inflation-restriction sequence yields the exact sequence
\[
\xymatrix{
0 \ar[r] & \mathrm{H}^1(G_{\mathbb{Q}_\ell}/I_\ell, M^{I_\ell}_2) 
\ar[r] & \mathrm{H}^1(\mathbb{Q}_\ell, M_2) 
\ar[r] & \mathrm{H}^1(I_\ell, M_2)^{G_{\mathbb{Q}_\ell}/I_\ell} \ar[r] & 0 .
}
\]
Here, we have
\begin{align*}
\mathrm{H}^1(G_{\mathbb{Q}_\ell}/I_\ell, M^{I_\ell}_2) & \simeq M^{I_\ell}_2 / (\mathrm{Frob}_\ell - 1) M^{I_\ell}_2 \\
& \simeq  M_2 / (\mathrm{Frob}_\ell - 1) M_2 \\
& \simeq 0
\end{align*}
where the first isomorphism follows from that $G_{\mathbb{Q}_\ell}/I_\ell$ is topologically generated by $\mathrm{Frob}_\ell$,
the second isomorphism follows from $M^{I_\ell}_2 = M_2$ (since $\ell \neq p$), and
the last isomorphism follows from that $\mathrm{Frob}_\ell - 1$ acts on $M_2$ as multiplication by $\ell - 1$ and $M_2$ is divisible.
Thus, we have $\mathrm{H}^1(\mathbb{Q}_\ell, M_2) \simeq \mathrm{H}^1(I_\ell, M_2)^{G_{\mathbb{Q}_\ell}/I_\ell}$.
Since $I_\ell$ acts on $M_2$ trivially, we have
\begin{align} \label{eqn:H^1_M_2}
\begin{split}
\mathrm{H}^1(\mathbb{Q}_\ell, M_2) & \simeq \mathrm{H}^1(I_\ell, M_2)^{G_{\mathbb{Q}_\ell}/I_\ell} \\
& \simeq  \mathrm{Hom}_{{G_{\mathbb{Q}_\ell}/I_\ell}}(I_\ell, M_2) \\
& \simeq \mathrm{Hom}_{{G_{\mathbb{Q}_\ell}/I_\ell}}(\mathbb{Z}_p(1), E/\mathcal{O}(1)) \\
& \simeq \mathrm{Hom}_{{G_{\mathbb{Q}_\ell}/I_\ell}}(\mathcal{O}(1), E/\mathcal{O}(1)) \\
& \simeq \mathrm{Hom}_{\mathcal{O}}(\mathcal{O}, E/\mathcal{O}) \\
& \simeq E/\mathcal{O}
\end{split}
\end{align}
From the short exact sequence
\[
\xymatrix{
0 \ar[r] & M_2 \ar[r]  & M_0 \ar[r]  & M_0/M_2 \ar[r] & 0 ,
}
\]
we obtain the exact sequence
\[
\xymatrix{
0 \ar[r] & \mathrm{H}^0(\mathbb{Q}_\ell,  M_2) \ar[r]  & \mathrm{H}^0(\mathbb{Q}_\ell,  M_0 ) \ar[r]  & \mathrm{H}^0(\mathbb{Q}_\ell,  M_0/M_2 ) \ar[r]^-{\delta} & \mathrm{H}^1(\mathbb{Q}_\ell,  M_2) \ar[r] & \mathrm{H}^1(\mathbb{Q}_\ell,  M_0) .
}
\]
By counting the size of each term as in (\ref{eqn:counting_H^0}) and (\ref{eqn:H^1_M_2}), the connecting map $\delta$ becomes the multiplication by $\lambda^{t}$ on $E/\mathcal{O}$ (up to a unit), which is surjective. Thus, 
the image of $\mathrm{H}^1(\mathbb{Q}_\ell,  M_2)$ in $\mathrm{H}^1(\mathbb{Q}_\ell,  M_0)$ vanishes.
It proves Proposition \ref{prop:local_computation}.(1) below.

Recall the inflation-restriction sequence
\[
\xymatrix{
0 \ar[r] & \mathrm{H}^1(G_{\mathbb{Q}_\ell}/I_\ell, M^{I_\ell}_0) 
\ar[r] & \mathrm{H}^1(\mathbb{Q}_\ell, M_0) 
\ar[r] & \mathrm{H}^1(I_\ell, M_0)^{G_{\mathbb{Q}_\ell}/I_\ell} \ar[r] & 0 .
}
\]
Since the action of $I_\ell$ on $M_0$ factors through the tame quotient $I^t_\ell$, we have 
$$\mathrm{H}^1(I_\ell, M_0) \simeq \mathrm{H}^1(I^t_\ell, M_0) \simeq M_0 / (\tau - 1)M_0$$
where $\tau$ is a topological generator of $I^t_\ell$.
By using Equation (\ref{eqn:ad^0_f_explicit_matrix}), we check $(\tau - 1)M_0  = M_1$.
Therefore, $\mathrm{H}^1(I_\ell, M_0) \simeq M_2/M_1 \simeq E/\mathcal{O}(-1)$.
Finally, we have $\mathrm{H}^1(I_\ell, M_0)^{G_{\mathbb{Q}_\ell}/I_\ell} \simeq \lambda^{-s}\mathcal{O}/\mathcal{O}$.

Note that 
$\mathrm{H}^1(G_{\mathbb{Q}_\ell}/I_\ell, M^{I_\ell}_0) \simeq  M^{I_\ell}_0 / (\mathrm{Frob}_\ell - 1)M^{I_\ell}_0$.
Then by using Equation (\ref{eqn:counting_H^0}) again, we have
$$M^{I_\ell}_0 \simeq E/\mathcal{O}(1) \oplus \lambda^{-t}\mathcal{O}/\mathcal{O} \oplus \lambda^{-t}\mathcal{O}/\mathcal{O}(-1)$$
as $G_{\mathbb{Q}_\ell}/I_\ell$-modules.
Thus, we also have
\begin{align*}
(\mathrm{Frob}_\ell - 1) M^{I_\ell}_0  & \simeq (\ell - 1) E/\mathcal{O} \oplus 0 \oplus (\ell^{-1} - 1) \lambda^{-t}\mathcal{O}/\mathcal{O} \\
 & \simeq E/\mathcal{O} \oplus 0 \oplus \lambda^{s-t}\mathcal{O}/\mathcal{O}
\end{align*}
as $\mathcal{O}$-modules.
Thus, $\mathrm{H}^1(G_{\mathbb{Q}_\ell}/I_\ell, M^{I_\ell}_0) \simeq \lambda^{-t}\mathcal{O}/\mathcal{O} \oplus \lambda^{-\mathrm{min}(s,t)}\mathcal{O}/\mathcal{O}$ as $\mathcal{O}$-modules, so Proposition \ref{prop:local_computation}.(2) below follows.
\begin{prop} \label{prop:local_computation}
Let $f$ be a newform such that $\rho_f \simeq \rho$ and $\rho \vert_{G_{\mathbb{Q}_\ell}}$ is an $\ell$-new deformation of $\overline{\rho} \vert_{G_{\mathbb{Q}_\ell}}$. Then we have the following statements:
\begin{enumerate}
\item The local condition
$$L_\ell = \mathrm{Im} \left( \mathrm{H}^1(\mathbb{Q}_\ell, M_2) \to \mathrm{H}^1(\mathbb{Q}_\ell, \mathrm{ad}^0(f) \otimes E/\mathcal{O}) \right)$$
vanishes.
\item $\mathrm{H}^1(\mathbb{Q}_\ell, \mathrm{ad}^0(f) \otimes E/\mathcal{O}))$ is a non-trivial finite $\mathcal{O}$-module and
$$\mathrm{length}_{\mathcal{O}} \mathrm{H}^1(\mathbb{Q}_\ell, \mathrm{ad}^0(f) \otimes E/\mathcal{O})) = \mathrm{ord}_\lambda(\ell - 1) + t_f(\ell) + \mathrm{min} (\mathrm{ord}_\lambda(\ell - 1), t_f(\ell)).$$
\end{enumerate}
\end{prop}
\begin{cor} \label{cor:local_computation}
Keep all the assumptions of Proposition \ref{prop:local_computation}.
If we further assume that $\ell \not\equiv 1 \pmod{p}$ for all primes $\ell$ dividing $N^-/N(\overline{\rho})^-$, then
$$\mathrm{length}_{\mathcal{O}} \mathrm{H}^1(\mathbb{Q}_\ell, \mathrm{ad}^0(f) \otimes E/\mathcal{O})) =  t_f(\ell).$$
\end{cor}

\subsection{Global computation} \label{subsec:global_computation}
Let $\pi_{f^{\Sigma}} : \mathbb{T}(N_{\overline{\rho},\Sigma}/r)_{ \mathfrak{m}_{\Sigma} } \to \mathcal{O}$ be the map associated to $f^{\Sigma}$ as defined in $\S$\ref{subsec:modularity_lifting_congruence_ideals}.
Let
 $\wp_{f^{\Sigma}} = \mathrm{ker} \left(\pi_{f^{\Sigma}} \right)$, and denote by
 $\Phi_{f^{\Sigma}}$ the cotangent space at $\wp_{f^{\Sigma}}$ and
 by $\eta_{f^{\Sigma}} (N_{\overline{\rho},\Sigma}/r)$ the congruence ideal.
Then we have
\begin{align*}
\mathrm{ord}_\lambda \eta_{f^{\Sigma}}(N_{\overline{\rho},\Sigma}/r)
& = \mathrm{length}_{\mathcal{O}} \Phi_{f^{\Sigma}}\\
& =  \mathrm{length}_{\mathcal{O}} \mathrm{Sel}_{\Sigma} (\mathbb{Q}, \mathrm{ad}^0(f) \otimes E/\mathcal{O})
\end{align*}
where the first equality follows from 
Theorem \ref{thm:diamond-flach-guo-dimitrov}, Proposition \ref{prop:complete_intersection_congruence_ideals}, and Remark \ref{rem:complete_intersection_congruence_ideals}
and the second equality follows from  (\ref{eqn:tangent-selmer}).

Let $\pi^{N^-}_{f^{\Sigma^+}} : \mathbb{T}(N^+_{\overline{\rho},\Sigma^+}/r \cdot N^-)^{N^-}_{ \mathfrak{m}_{\Sigma^+} } \to \mathcal{O}$ be the  $N^-$-new quotient map and $f^{\Sigma^+}$ as defined in proof of Theorem \ref{thm:refined_R=T}.
Let
 $\wp^{N^-}_{f^{\Sigma^+}} = \mathrm{ker} \left( \pi^{N^-}_{f^{\Sigma^+}} \right)$, and
 denote by $\Phi^{N^-}_{f^{\Sigma^+}}$ the cotangent space at $\wp^{N^-}_{f^{\Sigma^+}}$ and
by  $\eta_{f^{\Sigma^+}} (N^+_{\overline{\rho},\Sigma^+}/r, N^-)$ by the congruence ideal.
Then we have
\begin{align*}
 \mathrm{ord}_\lambda \eta_{f^{\Sigma^+}}(N_{\overline{\rho},\Sigma^+}/r, N^-) & = \mathrm{length}_{\mathcal{O}} \Phi^{N^-}_{f^{\Sigma^+}} \\
& = \mathrm{length}_{\mathcal{O}} \mathrm{Sel}^{\Sigma^-}_{\Sigma^+} (\mathbb{Q}, \mathrm{ad}^0(f) \otimes E/\mathcal{O}) 
\end{align*}
where the first equality follows from 
 Proposition \ref{prop:complete_intersection_congruence_ideals}, Remark \ref{rem:complete_intersection_congruence_ideals}, and Theorem \ref{thm:refined_R=T}
 and the second equality follows from (\ref{eqn:tangent-selmer}).

Now it suffices to compute the difference between 
$\mathrm{Sel}^{\Sigma^-}_{\Sigma^+} (\mathbb{Q}, \mathrm{ad}^0(f) \otimes E/\mathcal{O})$
and
$\mathrm{Sel}_{\Sigma} (\mathbb{Q}, \mathrm{ad}^0(f) \otimes E/\mathcal{O})$.
Note that $$\mathrm{Sel}^{\Sigma^-}_{\Sigma^+} (\mathbb{Q}, \mathrm{ad}^0(f) \otimes E/\mathcal{O}) \subseteq \mathrm{Sel}_{\Sigma} (\mathbb{Q}, \mathrm{ad}^0(f) \otimes E/\mathcal{O}) .$$
Consider the following ``small" Selmer group $\mathrm{Sel}_{\mathcal{L}} (\mathbb{Q}, \mathrm{ad}^0(f) \otimes E/\mathcal{O}) $
satisfying
\begin{itemize}
\item $L_p = \mathrm{H}^1_{\mathrm{ur}} (\mathbb{Q}_p, \mathrm{ad}^0(f) \otimes E/\mathcal{O})$,
\item $L_\ell = \mathrm{H}^1_{\mathrm{min}}(\mathbb{Q}_\ell, \mathrm{ad}^0(f) \otimes E/\mathcal{O})$ if $\ell \not\in \Sigma^+ \cup \Sigma^- \cup \lbrace  p \rbrace$,
\item $L_\ell = \mathrm{H}^1_{\mathrm{new}}(\mathbb{Q}_\ell, \mathrm{ad}^0(f) \otimes E/\mathcal{O})$ if $\ell \in \Sigma^-$, and
\item $L_\ell = \mathrm{H}^1_{\mathrm{ur}}(\mathbb{Q}_\ell, \mathrm{ad}^0(f) \otimes E/\mathcal{O})$ if $\ell \in \Sigma^+$.
\end{itemize}
Then we have
$$\mathrm{Sel}_{\mathcal{L}} (\mathbb{Q}, \mathrm{ad}^0(f) \otimes E/\mathcal{O}) \subseteq \mathrm{Sel}^{\Sigma^-}_{\Sigma^+} (\mathbb{Q}, \mathrm{ad}^0(f) \otimes E/\mathcal{O}) \subseteq \mathrm{Sel}_{\Sigma} (\mathbb{Q}, \mathrm{ad}^0(f) \otimes E/\mathcal{O})$$
and $\phi_{\mathcal{L}}$ is surjective by Proposition \ref{prop:surjectivity-of-Selmer} and the properties of local deformation conditions in $\S$\ref{subsec:local_deformation_conditions}.
By Proposition \ref{prop:comparison_selmer}, we have
\begin{align*}
\dfrac{\mathrm{Sel}^{\Sigma^-}_{\Sigma^+} (\mathbb{Q}, \mathrm{ad}^0(f) \otimes E/\mathcal{O})}{\mathrm{Sel}_{\mathcal{L}} (\mathbb{Q}, \mathrm{ad}^0(f) \otimes E/\mathcal{O})} & \simeq \dfrac{\mathrm{H}^1_{f}}{\mathrm{H}^1_{\mathrm{ur}}} \times \prod_{\ell \in \Sigma^+} \dfrac{\mathrm{H}^1}{\mathrm{H}^1_{\mathrm{ur}}} , \\
\dfrac{\mathrm{Sel}_{\Sigma} (\mathbb{Q}, \mathrm{ad}^0(f) \otimes E/\mathcal{O})}{\mathrm{Sel}_{\mathcal{L}} (\mathbb{Q}, \mathrm{ad}^0(f) \otimes E/\mathcal{O})} & \simeq \dfrac{\mathrm{H}^1_{f}}{\mathrm{H}^1_{\mathrm{ur}}} \times \prod_{\ell \in \Sigma^+} \dfrac{\mathrm{H}^1}{\mathrm{H}^1_{\mathrm{ur}}}
\times \prod_{\ell \in \Sigma^-} \dfrac{\mathrm{H}^1}{\mathrm{H}^1_{\mathrm{new}}} 
\end{align*}
where $\mathrm{H}^1_{f}$, $\mathrm{H}^1_{\mathrm{ur}}$, $\mathrm{H}^1$, and $\mathrm{H}^1_{\mathrm{new}}$ are obvious abbreviations. Thus, we have
$$\dfrac{\mathrm{Sel}_{\Sigma} (\mathbb{Q}, \mathrm{ad}^0(f) \otimes E/\mathcal{O})}{\mathrm{Sel}^{\Sigma^-}_{\Sigma^+} (\mathbb{Q}, \mathrm{ad}^0(f) \otimes E/\mathcal{O})} \simeq \prod_{\ell \in \Sigma^-} \dfrac{\mathrm{H}^1}{\mathrm{H}^1_{\mathrm{new}}} \simeq \prod_{\ell \in \Sigma^-} \mathrm{H}^1 $$
where the last isomorphism follows from Proposition \ref{prop:local_computation}.(1).
Combining all the results, we have
\begin{align*}
& \ \mathrm{ord}_{\lambda} \eta_{f^{\Sigma}}(N_{\overline{\rho},\Sigma} /r) - \mathrm{ord}_{\lambda} \eta_{f^{\Sigma^+}}(N^+_{\overline{\rho},\Sigma^+}/r \cdot N^-) \\
= & \  \mathrm{length}_{\mathcal{O}} \mathrm{Sel}_{\Sigma} (\mathbb{Q}, \mathrm{ad}^0(f) \otimes E/\mathcal{O})
 - \mathrm{length}_{\mathcal{O}} \mathrm{Sel}^{\Sigma^-}_{\Sigma^+} (\mathbb{Q}, \mathrm{ad}^0(f) \otimes E/\mathcal{O}) \\
= & \ \sum_{\ell \in \Sigma^-} \mathrm{length}_{\mathcal{O}} \mathrm{H}^1 (\mathbb{Q}_\ell, \mathrm{ad}^0(f) \otimes E/\mathcal{O}) 
\end{align*}
and Theorem \ref{thm:key_thm} follows from Proposition \ref{prop:local_computation}.(2) and Corollary \ref{cor:local_computation}.
\section{Congruence ideals and adjoint $L$-values} \label{sec:adjoint_L-values}
The goal of this section is to prove Proposition \ref{prop:second_correction} (the $L$-value computation).  
\subsection{Adjoint $L$-functions}
Let $f \in S_k(\Gamma_0(N))$ be a newform.
Then the $L$-function of $f$ is defined by
$$L(f,s) := \prod_{\ell \nmid N} (1 - a_\ell(f) p^{-s} + p^{k-1-2s})^{-1} \cdot \prod_{\ell \mid N} (1 - a_\ell(f) p^{-s})^{-1} $$
and note that the Euler factor $L_\ell (f,s)$ at $\ell \nmid N$ is
$$(1 - a_\ell(f) p^{-s} + p^{k-1-2s})^{-1} = (1 - \alpha_\ell(f) p^{-s})^{-1} \cdot (1 - \beta_\ell(f) p^{-s})^{-1} .$$
For a prime $\ell \neq p$, we recall
\[
\xymatrix{
c_0(\ell) := \mathrm{ord}_\ell(N), & d_0(\ell) := \mathrm{dim}_{E} \ \mathrm{H}^1( I_\ell, V_f).
}
\]
Then we have
$$d_0(\ell) = \left\lbrace \begin{array}{ll} 
2 & \textrm{if } \ \ell \nmid N , \\
1 & \textrm{if } \ \ell \mid N \textrm{ and } a_\ell(f) \neq 0 , \\
0 & \textrm{if } \ \ell \mid N \textrm{ and } a_\ell(f) = 0 .
\end{array} \right.$$
If $d_0(\ell) = 2$, then 
$$L_\ell(\mathrm{ad}^0(f), s) := \left( 1 - \frac{\alpha_\ell(f)}{\beta_\ell(f)} p^{-s} \right)^{-1} \cdot
(1-\ell^{-s})^{-1} \cdot
\left( 1 - \frac{\beta_\ell(f)}{\alpha_\ell(f)} p^{-s} \right)^{-1} .$$
If $d_0(\ell) = 1$, then
$$L_\ell(\mathrm{ad}^0(f), s) := \left\lbrace \begin{array}{ll} 
(1-\ell^{-1-s})^{-1} & \textrm{if } \ \pi_\ell(f) \textrm{ is special}, \\
(1-\ell^{-s})^{-1} & \textrm{if } \ \pi_\ell(f) \textrm{ is principal series}
\end{array} \right.$$
where $\pi_\ell(f)$ is the local component of the automorphic representation attached to $f$ at $\ell$.
We omit the definition of $L_\ell(\mathrm{ad}^0(f), s)$ when $d_0(\ell) = 0$, but note that it may not be 1 in general (``the exceptional set for $f$" cf.~\cite[$\S$1.7.2 and $\S$1.8.1]{diamond-flach-guo}).
For the exact calculation of congruence ideals, we modify Euler factors as follows:
$$L^{\mathrm{nv}}_\ell(\mathrm{ad}^0(f), s) := \left\lbrace \begin{array}{ll} 
L_\ell(\mathrm{ad}^0(f), s) & \textrm{if } \ d_0(\ell) >0, \\
1 & \textrm{if } \  d_0(\ell) =0.
\end{array} \right.$$
\subsection{Cohomology congruence ideals and their variants} \label{subsec:cohomology_congruence_ideals}
We quickly review cohomology congruence ideals following \cite[$\S$4.4]{ddt}, \cite{diamond-euler}, and \cite[$\S$1.7.3]{diamond-flach-guo}.

Let $k$ be an integer with $2 \leq k \leq p-1$.
Let $\overline{\rho}$ be the fixed irreducible residual representation and $f \in S_k(\Gamma_0(N))$ an eigenform with $\overline{\rho}_f \simeq \overline{\rho}$.
Let $\mathbb{T}(N)_{\mathfrak{m}_f}$ be the full Hecke algebra localized at the non-Eisenstein maximal ideal $\mathfrak{m}_f$.
As in $\S$\ref{subsubsec:level-str-Selmer}, it follows from \cite[Lemma 5.4.(iii)]{dimitrov-ihara} that $\mathbb{T}(N)_{\mathfrak{m}_f}$ is isomorphic to the image of $\mathbb{T}^S$  in $\mathrm{End}_{\mathcal{O}} ( \mathrm{H}^1_{\mathrm{\acute{e}t},c}(Y_1(N)_{\overline{\mathbb{Q}}}, \mathcal{F}^k_p)[\langle-\rangle - \mathbf{1}]_{\mathfrak{m}_f} )$.

We write $\mathcal{M}_{N} = \mathrm{H}^1_{\mathrm{\acute{e}t},c}(Y_1(N)_{\overline{\mathbb{Q}}}, \mathcal{F}^k_p)_{\mathfrak{m}_f}$.
We do not care too much about the compactly supported or parabolic cohomologies since $\mathfrak{m}_f$ is non-Eisenstein.

The cup product and Poincar\`{e} duality on Betti cohomology and the comparison between Betti and \'{e}tale cohomologies imply that there exists an alternating and non-degenerate pairing
\[
\xymatrix@R=0em{
\mathcal{M}_{N} \times \mathcal{M}_{N} \ar[r] & E \\
(x, y) \ar@{|->}[r] & x \cup w \cdot y
}
\]
where $w$ is the Atkin--Lehner involution, and
we denote the image of the pairing by 
$\mathcal{L} := \wedge^2 \mathcal{M}_{N} \subseteq E $.
We scale the pairing to have $\mathcal{L} = \mathcal{O}$ for convenience.
Then the pairing
\begin{equation} \label{eqn:perfect_pairing}
\mathcal{M}_{N} \times \mathcal{M}_{N} \to  \mathcal{O} 
\end{equation}
is perfect, and we define the \textbf{cohomology congruence ideal of $f$} by the image of the pairing
$$\delta_f(N) := \wedge^2 \mathcal{M}_{N}[\wp_f] \subseteq \mathcal{O} .$$
In order to identify congruence ideals and cohomology congruence ideals, we need the following freeness result.
\begin{prop} \label{prop:gorenstein}
Suppose that $\mathfrak{m}_f$ is non-Eisenstein.
Then
\begin{enumerate}
\item 
$\mathbb{T}(N)_{\mathfrak{m}_f}$ is Gorenstein, and
\item 
$\mathcal{M}_{N}$ is free over $\mathbb{T}(N)_{\mathfrak{m}_f}$ of rank two.
\end{enumerate}
\end{prop}
\begin{proof}
See \cite[Proposition 5.5]{dimitrov-ihara} (cf. \cite[Theorem 2.1]{faltings-jordan}, \cite[Theorem 1.13]{vatsal-cong}).
\end{proof}
\begin{rem}
The freeness obtained from Diamond's Taylor--Wiles system argument requires that the level should be of the form $N_{\overline{\rho},\Sigma}$ for some $\Sigma$ \cite[Theorem 3.4]{diamond_multiplicity_one}.
\end{rem}

Let $\mathfrak{a}^{N^-} = \mathrm{ker}( \mathbb{T}(N)_{\mathfrak{m}_f} \to \mathbb{T}(N)^{N^-}_{\mathfrak{m}_f} )$ and we define 
$\mathcal{M}^{N^-}_{N} :=
\mathcal{M}_{N}[\mathfrak{a}^{N^-}] $. 
\begin{lem} \label{lem:torsion_free}
Suppose that $\mathfrak{m}_f$ is non-Eisenstein.
The quotient $\mathcal{M}_{N} / \left( \mathcal{M}_{N}[\mathfrak{a}^{N^-}] \right)$ is torsion-free.
\end{lem}
\begin{proof}
It is straightforward from Proposition \ref{prop:gorenstein} and the torsion-freeness of $\mathcal{O}$.
\end{proof}


\begin{cor} \label{cor:N^--freeness}
Suppose that $\mathfrak{m}_f$ is non-Eisenstein and $f$ is new at all primes dividing $N^-$. 
Then $\mathcal{M}^{N^-}_{N}$ is free over $\mathbb{T}(N)^{N^-}_{\mathfrak{m}_f}$ of rank two and $\mathbb{T}(N)^{N^-}_{\mathfrak{m}_f}$ is Gorenstein.
\end{cor}
\begin{proof}
By Lemma \ref{lem:torsion_free}, we have $$\mathcal{M}^{N^-}_{N} \otimes_{\mathcal{O}} \mathbb{F} \subseteq \mathcal{M}_{N} \otimes_{\mathcal{O}} \mathbb{F} .$$
Since $S_k(\Gamma_0(N))^{N^-}_{\mathfrak{m}_f} \neq 0$, we have
$$
0 \neq \left( \mathcal{M}^{N^-}_{N} \otimes_{\mathcal{O}} \mathbb{F} \right) [\mathfrak{m}_f]  
\subseteq 
\left( \mathcal{M}_{N} \otimes_{\mathcal{O}} \mathbb{F} \right) [\mathfrak{m}_f] 
\simeq \mathbb{F}^{\oplus 2} 
$$
and all are even dimensional. Thus, we have
$\left( \mathcal{M}^{N^-}_{N} \otimes_{\mathcal{O}} \mathbb{F} \right) [\mathfrak{m}_f]
\simeq \mathbb{F}^{\oplus 2} $.
Note that
$\mathcal{M}^{N^-}_{N}$ is free over $\mathcal{O}$ and
 $\mathrm{dim}_{E} \left( \mathcal{M}^{N^-}_{N} \otimes_{\mathcal{O}} E \right)
= 2 \cdot \mathrm{dim}_{E} \left( \mathbb{T}(N)^{N^-}_{\mathfrak{m}_f} \otimes_{\mathcal{O}} E \right)$.
Thus, $\mathcal{M}^{N^-}_{N}$ is free of rank two over $\mathbb{T}(N)^{N^-}_{\mathfrak{m}_f}$.
The Gorenstein property immediately follows from the freeness.
\end{proof}
\begin{rem}
Since we work in the subspace of classical modular forms, we do not impose any condition on $N^-$ here.
If we use modular forms on Shimura curves or Hida varieties attached to the quaternion algebra of discriminant $N^-$ via the Jacquet--Langlands correspondence, then we should assume  $\ell \not\equiv \pm 1 \pmod{p}$ for all residually unramified prime $\ell$ dividing $N^-$ at least (e.g. Assumption \ref{assu:condition_CR+}).
\end{rem}
From the perfect pairing (\ref{eqn:perfect_pairing}), we have a perfect pairing
$$\mathcal{M}_{N}[\mathfrak{a}^{N^-}] \times \left( \mathcal{M}_{N}  \otimes_{\mathbb{T}(N)_{\mathfrak{m}_f} } \mathbb{T}(N)^{N^-}_{\mathfrak{m}_f} \right) \to \mathcal{O} .$$
By Corollary \ref{cor:N^--freeness}, the perfect pairing can be written as
\begin{equation} \label{eqn:perfect_pairing_N^-}
\mathcal{M}^{N^-}_{N} \times \mathcal{M}^{N^-}_{N} \to \mathcal{L}^{N^-}
\end{equation}
where 
$\mathcal{L}^{N^-} := \wedge^2 \mathcal{M}^{N^-}_{N} = \lambda^{t(N^+,N^-)} \mathcal{O}$ for some $t(N^+,N^-)$ (comparing with $\mathcal{L} = \mathcal{O}$).
Then we define the \textbf{$N^-$-new cohomology congruence ideal of $f$} by
$$\delta_f(N^+, N^-) := \lambda^{-t(N^+,N^-)} \cdot \delta_f(N) .$$
Consider the map
$$\gamma: \mathcal{M}_{N} \to \mathcal{M}_{N^{\Sigma}_f}$$
defined in \cite[$\S$1.7.3]{diamond-flach-guo}. Under the comparison between de Rham and \'{e}tale cohomologies, this map corresponds to the map $f \mapsto f^{\Sigma}$ in the de Rham side.
Then we define the \textbf{cohomology congruence ideal of $f^{\Sigma}$} by
$$\delta_{f^{\Sigma}} (N^\Sigma_f) := \wedge^2 \mathcal{M}_{N^{\Sigma}_f}[\wp_{f^\Sigma}] \subseteq \mathcal{L}_\Sigma := \wedge^2 \mathcal{M}_{N^{\Sigma}_f} = \mathcal{O}$$
where the last identification is the normalization of $\mathcal{L}_\Sigma$ as before.
\begin{prop}
Suppose that $\overline{\rho}$ is irreducible. Then
$$\delta_{f^\Sigma} (N^\Sigma_f) = \delta_{f} (N) \cdot \prod_{\ell \in \Sigma} L^{\mathrm{nv}}_\ell(\mathrm{ad}^0(f), 1)^{-1} .$$
\end{prop}
\begin{proof}
See \cite[Proposition 1.4.(c)]{diamond-flach-guo}.
\end{proof}
Decompose $N = N^+ \cdot N^-$ and suppose that $\Sigma^+ = \Sigma$ does not contain any divisor of $N^-$.
Consider the restriction of $\gamma$ to the $N^-$-new parts
$$\gamma^{N^-} : \mathcal{M}^{N^-}_{N} \to \mathcal{M}^{N^-}_{N^{\Sigma}_f} .$$
Let $\mathcal{L}^{N^-}_\Sigma := \wedge^2 \mathcal{M}^{N^-}_{N^{\Sigma}_f} = \lambda^{t(N^+_{\overline{\rho},\Sigma^+}/r, N^-)} \mathcal{O} \subseteq \mathcal{L}_\Sigma =  \mathcal{O}$
for some $t(N^+_{\overline{\rho},\Sigma^+}/r, N^-)$.
In the same manner, we define the \textbf{$N^-$-new cohomology congruence ideal of $f^{\Sigma^+}$} by
$$\delta_{f^{\Sigma^+}} (N^+_{\overline{\rho},\Sigma^+}/r, N^-) := \lambda^{-t(N^+_{\overline{\rho},\Sigma^+}/r, N^-)} \cdot \delta_{f^{\Sigma^+}} (N^+_{\overline{\rho},\Sigma^+}/r \cdot N^-) .$$
\begin{cor} \label{cor:N^--new-adjoint-L-values}
Suppose that $\overline{\rho}$ is irreducible. Then
$$\delta_{f^{\Sigma^+}} (N^+_{\overline{\rho},\Sigma^+} / r, N^-) = \delta_{f} (N^+, N^-) \cdot \prod_{\ell \in \Sigma^+} L^{\mathrm{nv}}_\ell(\mathrm{ad}^0(f), 1)^{-1} .$$
Thus, $t(N^+_{\overline{\rho},\Sigma^+}/r, N^-)= t(N^+, N^-)$.
\end{cor}
\begin{proof}
The same proof as in \cite[Proposition 1.4.(c)]{diamond-flach-guo} directly works.
\end{proof}
\begin{rem} 
\begin{enumerate}
\item Indeed, $t(N^+, N^-)$ becomes $\sum_{q \vert N^-} t_f(q)$ at the end under our assumptions. Thus, $t(N^+_{\overline{\rho},\Sigma^+}/r, N^-)= t(N^+, N^-)$ can be expected easily.
\item Cohomology congruence ideals are stable under base change \cite[$\S$1.7.3]{diamond-flach-guo}.
\end{enumerate}
\end{rem}
\begin{cor} \label{cor:delta-eta-identification}
Suppose that $\mathfrak{m}_f$ is non-Eisenstein.
\begin{enumerate}
\item $\delta_f(N) = \eta_f(N) $, and
\item $\delta_f(N^+, N^-) = \eta_f(N^+, N^-) $.
\end{enumerate}
\end{cor}
\begin{proof}
See \cite[Lemma 4.17]{ddt} (cf. \cite[Lemma 3.1]{diamond-euler}).
\end{proof}
Thus, Proposition \ref{prop:second_correction} immediately follows.

\section{Arithmetic applications} \label{sec:appendix}
In this section, we briefly describe various arithmetic applications of Theorem \ref{thm:main_thm}.

In $\S$\ref{subsec:gross_periods} and $\S$\ref{subsec:mu-invariants},
we generalize the work of Pollack--Weston \cite{pw-mu} and Chida--Hsieh \cite{chida-hsieh-p-adic-L-functions} on the comparison between Hida's canonical periods and Gross periods to higher weight modular forms (Corollary \ref{cor:gross-period-comparison}).
We also extend the $\mu$-part of the anticyclotomic main conjecture for modular forms of higher weight to Greenberg Selmer groups (Corollary \ref{cor:mu-part-imc}).

In $\S$\ref{subsec:periods-shimura-curves}, we study the integral periods arising from Shimura curves.
Under certain assumptions, we show that if the canonical periods arising from modular curves and Shimura curves differ by a $p$-adic unit, then Prasanna's conjecture holds.
See Corollary \ref{cor:prasanna-conjecture} for the  precise  statement.


\subsection{Comparison with Gross periods} \label{subsec:gross_periods}
We recall \cite[$\S$6]{chida-hsieh-p-adic-L-functions} with some modifications.
Let $f \in S_k(\Gamma_0(N))$ be a newform as in Theorem \ref{thm:main_thm}.
We keep the following assumption in this subsection.
\begin{assu} \label{assu:definite}
Assume that $N^-$ is the square-free product of an \emph{odd} number of primes.
\end{assu}
\subsubsection{Quaternionic modular forms}
Let $B$ be the definite quaternion algebra over $\mathbb{Q}$ of discriminant $N^-$ and $R$ an Eichler order of level $N^+$.
Let $f_B$ be a Jacquet--Langlands transfer of $f$, i.e. a continuous function
$$f_B : B^{\times} \backslash \widehat{B}^{\times} / \widehat{R}^{(p), \times} \to \mathrm{Sym}^{k-2}(E^2)$$
such that $f_B(a \cdot g \cdot r) =  r^{-1} \circ f_B(g)$ for $a \in B^{\times}$ and $r \in R^\times_p \simeq \mathrm{GL}_2(\mathbb{Z}_p)$ and the Hecke eigenvalues of $f$ and $f_B$ are same at all primes not dividing $N^-$.
Denote by
$$S^{N^-}_k(R, E) =\left\lbrace
f: B^{\times} \backslash \widehat{B}^{\times} / \widehat{R}^{(p), \times} \to \mathrm{Sym}^{k-2}(E^2)
: f(a \cdot g \cdot r) =  r^{-1} \circ f(g), f \textrm{ is not constant}
 \right\rbrace .$$
We recall the following simple form of the Jacquet--Langlands correspondence.
\begin{thm} \label{thm:jacquet-langlands-rational}
There exists a non-canonical isomorphism 
$$S_k(\Gamma_0(N), \mathbb{C})^{N^-} \simeq S^{N^-}_k(R, \mathbb{C}) $$
of Hecke modules with identification their Hecke algebras over $\mathbb{C}$.
\end{thm}
\begin{proof}
See \cite[Theorem 2.30 in $\S$2.3.6]{hida-hilbert} for this form of the correspondence.
Indeed, the isomorphism should be understood as the one-to-one correspondence between Hecke eigensystems.
\end{proof}
By using fixed embeddings $\iota_p$ and $\iota_\infty$, we integrally normalize each quaternionic modular form $f_B$ by its mod $p$ non-vanishing of the values of $f_B$ as in \cite[$\S$4.1]{chida-hsieh-p-adic-L-functions}.
The integral normalization of classical modular forms is given by the $q$-expansion.
Using these two integral structures,
 we are able to identify the Hecke modules
$$S_k(\Gamma_0(N), \mathcal{O})^{N^-} \simeq S^{N^-}_k(N, \mathcal{O}) $$
with identification of their Hecke algebras over $\mathcal{O}$; however, this identification itself is ad hoc. It will have a precise meaning after establishing the freeness of the quaternionic Hecke modules (Theorem \ref{thm:freeness_quaternion}).

\subsubsection{Quaternionic congruence ideals and Gross periods}
Consider the perfect pairing 
\begin{equation} \label{eqn:quaternionic_pairing}
\langle -, - \rangle_{N} : S^{N^-}_k(N, \mathcal{O}) \times S^{N^-}_k(N, \mathcal{O}) \to \mathcal{O}
\end{equation}
defined in \cite[(6.1)]{chida-hsieh-p-adic-L-functions}:
$$\langle f_B, g_B \rangle_{N} := \sum_{[b]} \langle f_B(b), g_B(bw) \rangle_k \cdot \left(\# (B^\times \cap b\widehat{R}^\times b^{-1}\widehat{\mathbb{Q}}^\times ) / \mathbb{Q}^\times \right)^{-1}$$
where
$w$ is the Atkin--Lehner operator for level $N^+$,
$[b]$ runs over $B^\times \backslash \widehat{B}^\times / \widehat{R}^\times\widehat{\mathbb{Q}}^\times$,
and $\langle - , - \rangle_k : \mathrm{Sym}^{k-2} (\mathcal{O}) \times \mathrm{Sym}^{k-2} (\mathcal{O}) \to \mathcal{O}$ is the perfect pairing defined in \cite[$\S$2.3]{chida-hsieh-p-adic-L-functions}.

Let $\xi_{f_B}(N^+, N^-)$ be the quaternionic analogue of the cohomology congruence ideal for $f_B$ using the above pairing (\ref{eqn:quaternionic_pairing}) as in \cite[$\S$2.1]{pw-mu} and \cite[(3.9) and (4.3)]{chida-hsieh-p-adic-L-functions}.

\textbf{Hida's canonical period for $f$} is defined by 
$$\Omega_f := \dfrac{(4\pi)^k \langle f, f\rangle_{\Gamma_0(N)}}{\eta_f(N)}$$
and the \textbf{Gross period for $f$} is defined by
$$\Omega^{N^-}_f := \dfrac{(4\pi)^k \langle f, f\rangle_{\Gamma_0(N)}}{\xi_{f_B}(N^+, N^-)} $$
 where $\langle -, -\rangle_{\Gamma_0(N)}$ is the Petersson inner product.
\subsubsection{The freeness of quaternionic Hecke modules and the comparison of periods}
\begin{assu}[CR$^+$] \label{assu:condition_CR+}
Suppose that $\overline{\rho}$ satisfies Assumption \ref{assu:working_assumptions}.
Let $f \in S_k(\Gamma_0(N))$ be a newform with decomposition $N = N^+ \cdot N^-$ such that
\begin{enumerate}
\item $\overline{\rho}_f \simeq \overline{\rho}$,
\item $2 \leq k < p-1$ and $p \nmid N$,
\item $p$, $N^+$, and $N^-$ are pairwisely relatively prime,
\item $N^-$ is square-free and the product of an odd number of primes,
\item If $q$ divides $N^-$ and $q \equiv \pm 1 \pmod{p}$, then $q$ divides $N(\overline{\rho})$,
\item If $q$ divides $N^+$ exactly and $q \equiv 1 \pmod{p}$, then $q$ divides $N(\overline{\rho})$, and
\item $(N(\overline{\rho}), N/N(\overline{\rho}))=1$.
\end{enumerate}
\end{assu}
\begin{rem}
Assumption \ref{assu:condition_CR+} is more strict than the assumptions in Theorem \ref{thm:main_thm}.
\end{rem}
The following theorem follows from the standard Taylor--Wiles sytem argument.
\begin{thm} \label{thm:freeness_quaternion}
Under Assumption \ref{assu:condition_CR+}, 
$S^{N^-}_k(N, \mathcal{O})_{\mathfrak{m}_f}$ is a free $\mathbb{T}(N)^{N^-}_{\mathfrak{m}_f}$-module of rank one.
\end{thm}
\begin{proof}
See \cite[Proposition 6.8]{chida-hsieh-p-adic-L-functions}.
\end{proof}
\begin{rem} 
\begin{enumerate}
\item 
In \cite[(D3) in $\S$6.2]{chida-hsieh-p-adic-L-functions}, the ordinary deformation condition at $p$ is considered only; however, the replacement of the ordinary deformation condition by the low weight crystalline deformation condition does not affect any of result in \cite[$\S$6]{chida-hsieh-p-adic-L-functions}.
It is well-known that the low weight crystalline deformation condition fits well with the Taylor--Wiles system argument \cite{diamond-flach-guo, dimitrov-ihara}.
\item Considering the local deformation condition at a prime $q$ exactly dividing $N/N(\overline{\rho})$ as in \cite[(D4) in $\S$6.2]{chida-hsieh-p-adic-L-functions}, the tame level $N^+_{\overline{\rho},\Sigma^+}$ in the Taylor--Wiles system argument can be directly chosen as $N^+$ under Assumption \ref{assu:condition_CR+}.(6).
\end{enumerate}
\end{rem}
We obtain the following comparison of two different periods.
\begin{cor} \label{cor:gross-period-comparison}
Under Assumption \ref{assu:condition_CR+}, 
the following statements are valid.
\begin{enumerate}
\item $\eta_f(N^+ ,N^-) = \xi_{f_B}(N^+, N^-) $.
\item $\mathrm{ord}_\lambda (  \Omega^{N^-}_f / \Omega_f ) = \mathrm{ord}_\lambda ( \dfrac{\eta_f(N)}{\eta_f(N^+, N^-)} ) = \sum_{q \vert N^-} t_f(q)$.
\end{enumerate}
\end{cor}
\begin{proof}
The first statement follows from Theorem \ref{thm:jacquet-langlands-rational} and \cite[Lemma 4.17]{ddt}.
The proof of \cite[Proposition 6.1]{chida-hsieh-p-adic-L-functions} applies to the second statement in the exactly same manner.
\end{proof}
\subsection{Anticyclotomic $\mu$-invariants of modular forms} \label{subsec:mu-invariants}
The goal of this subsection is to prove the $\mu$-part of the anticyclotomic main conjecture for modular forms of higher weight as an application of Corollary \ref{cor:gross-period-comparison} (so of Theorem \ref{thm:main_thm}). We only give a sketch of the argument here. See \cite{pw-mu, chida-hsieh-main-conj, chida-hsieh-p-adic-L-functions} for details.

We keep Assumption \ref{assu:definite} and Assumption \ref{assu:condition_CR+} in this subsection.
Let $K$ be an imaginary quadratic field such that 
$(\mathrm{disc}(K) , Np) = 1$ such that
\begin{itemize}
\item if a prime $\ell$ divides $N^-$, then $\ell$ is inert in $K$, and
\item if a prime  $\ell$ divides $N^+$, then $\ell$ splits in $K$.
\end{itemize}
\begin{assu} \label{assu:ordinary-PO}
Assume that $f$ is ordinary at $p$ and $a_p(f) \not\equiv \pm 1 \pmod{p}$.
\end{assu}
Following \cite{pw-mu, chida-hsieh-p-adic-L-functions},
we are able to define two slightly different anticyclotomic $p$-adic $L$-functions $L_p(K_\infty, f)$ and $\mathfrak{L}_p(K_\infty, f)$ of $(f,K_\infty/K)$ in $\Lambda = \mathcal{O}\llbracket \mathrm{Gal}(K_\infty/K) \rrbracket$ relative to the Gross period and the Hida's canonical period, respectively.
Then by Corollary \ref{cor:gross-period-comparison} (cf. \cite[$\S$2.2]{pw-mu}),
we have
 $$\mathfrak{L}_p(K_\infty, f) = L_p(K_\infty, f) \cdot  \dfrac{\eta_f(N)}{\xi_{f_B}(N^+, N^-)} .$$ 
 
Let $A^*_f$ be the central critical twist of $A_f$.
The minimal Selmer group 
$\mathrm{Sel}(K_\infty, A^*_f)$ is defined as the kernel of the map
$$\mathrm{H}^1(K_\infty, A^*_f)  \to \prod_{w \nmid p} \mathrm{H}^1(K_{\infty, w}, A^*_f) \times \prod_{w \mid p} \dfrac{\mathrm{H}^1(K_{\infty, w}, A^*_f)}{\mathrm{H}^1_{\mathrm{ord}}(K_{\infty, w}, A^*_f)}$$
and the Greenberg Selmer group
$\mathfrak{Sel}(K_\infty, A^*_f)$ is defined as the kernel of the map
$$\mathrm{H}^1(K_\infty, A^*_f)  \to \prod_{w \nmid p} \mathrm{H}^1(I_{\infty, w}, A^*_f) \times \prod_{w \mid p} \dfrac{\mathrm{H}^1(K_{\infty, w}, A^*_f)}{\mathrm{H}^1_{\mathrm{ord}}(K_{\infty, w}, A^*_f)}$$
where $w$ runs over all places of $K_\infty$, $I_{\infty, w}$ is the inertia subgroup at $w$ and $\mathrm{H}^1_{\mathrm{ord}}$ is the standard ordinary condition.

For $\Lambda$-module $M = \mathrm{Sel}(K_\infty, A^*_f)$ or $\mathfrak{Sel}(K_\infty, A^*_f)$, we denote by
$\mu (M)$ and $\lambda (M)$ by the $\mu$-invariant and the $\lambda$-invariant of the characteristic ideal of the Pontryagin dual of $M$, respectively.
\begin{prop}[Pollack--Weston] \label{prop:pollack-weston-comparison}
Under the assumptions in Theorem \ref{thm:main_thm}, $k=2$, and Assumption \ref{assu:ordinary-PO}, we have
\begin{align*}
\lambda (\mathrm{Sel}(K_\infty, A^*_f)) & = \lambda ( \mathfrak{Sel}(K_\infty, A^*_f) ) , \\
\mu (\mathfrak{Sel}(K_\infty, A^*_f))  & = \mu (\mathrm{Sel}(K_\infty, A^*_f))+ \sum_{q \vert N^-} t_f(q) .
\end{align*}
\end{prop}
\begin{proof}
This is \cite[Corollary 5.2]{pw-mu}.
\end{proof}
\begin{thm}[Chida--Hsieh] \label{thm:chida-hsieh-mu}
Under Assumption \ref{assu:condition_CR+} and Assumption \ref{assu:ordinary-PO}, we have 
$$\mu ( L_p(K_\infty, f) ))=  \mu ( \mathrm{Sel}(K_\infty, A^*_f) ) =0.$$
\end{thm}
\begin{proof}
See \cite[Theorem C]{chida-hsieh-p-adic-L-functions} and \cite[Corollary 1]{chida-hsieh-main-conj} with an enhancement by \cite{kim-pollack-weston}.
\end{proof}
We extend Theorem \ref{thm:chida-hsieh-mu} to Greenberg Selmer groups and anticyclotomic $p$-adic $L$-functions  relative to Hida's canonical periods.
\begin{cor} \label{cor:mu-part-imc}
Under Assumption \ref{assu:condition_CR+} and Assumption \ref{assu:ordinary-PO}, we have 
 $$\mu ( \mathfrak{L}_p(K_\infty, f) ) = \mu ( \mathfrak{Sel}(K_\infty, A^*_f) ) = \sum_{q \vert N^-} t_f(q).$$
\end{cor}
\begin{proof}
It is immediate from Corollary \ref{cor:gross-period-comparison}, Proposition \ref{prop:pollack-weston-comparison}, and Theorem \ref{thm:chida-hsieh-mu}.
\end{proof}
Corollary \ref{cor:mu-part-imc} completes a higher weight generalization of the $\mu$-part of the anticyclotomic main conjecture \cite[Theorem 6.9]{pw-mu}.
It seems possible to prove the supersingular analogue of Corollary \ref{cor:mu-part-imc} when $a_p(f) = 0$ as in \cite{pw-mu}, but we omit it.

\subsection{Comparison with integral periods of Shimura curves} \label{subsec:periods-shimura-curves}
Let $f \in S_k(\Gamma_0(N))$ be a newform as in Theorem \ref{thm:main_thm}.
\begin{assu} \label{assu:indefinite}
Assume that $N^-$ is the square-free product of an \emph{even} number of primes.
\end{assu}
Let $B$ be the \emph{indefinite} quaternion algebra over $\mathbb{Q}$ of discriminant $N^-$ and $f_B$ the Jacquet--Langlands transfer of $f$.
\subsubsection{Integrality of automorphic forms and the freeness of the Hecke modules}
An integral normalization of $f_B$ is much more delicate than the definite case since the geometry of Shimura curves is substantially involved and the normalization via the $q$-expansion is not available due to the lack of cusps.
In \cite{prasanna-annals}, $f_B$ is $p$-integrally normalized by considering the minimal regular model of the corresponding Shimura curve over $\mathbb{Z}_p$. It can also be checked by considering the values of $f_B$ at CM points via \cite[Proposition 2.9]{prasanna-annals}. We assume the integrality in this subsection.
\begin{assu} \label{assu:integral-shimura}
The Jacquet--Langlands transfer $f_B$ of $f$ is integrally normalized in the sense of \cite{prasanna-annals}.
\end{assu}
Let $X^{N^-}(N^+)$ be the Shimura curve of level $N^+$ and discriminant $N^-$ over $\mathbb{Q}$ and
$$\mathcal{M}'_{N^+} := \mathrm{H}^1_{\mathrm{\acute{e}t}}(X^{N^-}(N^+)_{\overline{\mathbb{Q}}}, \mathcal{F}^k_p)_{\mathfrak{m}_{f_B}}$$ the cohomology of $X^{N^-}(N^+)$ localized at the non-Eisenstein maximal ideal $\mathfrak{m}_{f_B}$.

We make the following freeness assumption.
\begin{assu} \label{assu:freeness-shimura}
The Hecke module $\mathcal{M}'_{N^+}$ is free of rank two over $\mathbb{T}(N)^{N^-}_{\mathfrak{m}_f}$.
\end{assu}
\begin{rem}
We discuss the known freeness result of $\mathcal{M}'_{N^+}$ in $\S$\ref{subsubsec:remark-freeness-shimura-curves}.
Since it is not established in full generality, we keep it as an assumption.
\end{rem}
It is possible to pin down the canonical periods of $f_B$ via the Eichler--Shimura isomorphism for Shimura curves (e.g.~\cite[(3.6)]{saito-hilbert}) under Assumption \ref{assu:freeness-shimura}.
Using the Poincar\`{e} duality and the Betti-\'{e}tale comparison as explained in $\S$\ref{subsec:cohomology_congruence_ideals}, 
there exists a (surjective) perfect pairing
$$\langle -, - \rangle' : \mathcal{M}'_{N^+} \times \mathcal{M}'_{N^+} \to \mathcal{O} .$$
and we are able to define the following analogue of the cohomology congruence ideals for Shimura curves by
$$\xi_{f_B}(N^+, N^-) := \wedge^2 \mathcal{M}'_{N^+}[\wp_{f_B}] \subseteq \mathcal{O}.$$ 
Let $S_k(\Gamma_0(N), \mathcal{O})^{N^-}_{\mathfrak{m}_f} \subseteq S_k(\Gamma_0(N), \mathcal{O})_{\mathfrak{m}_f}$ be the $\mathcal{O}$-submodule generated by
normalized eigenforms in $S_k(\Gamma_0(N))$ which are $N^-$-new and congruent to $f$ modulo $\lambda$. 
Let $\mathcal{M}^{N^-}_{N} \subseteq \mathcal{M}_{N}$ be the Hecke-stable submodule generated by the image of the two copies of $S_k(\Gamma_0(N), \mathcal{O})^{N^-}_{\mathfrak{m}_f}$ under the integral Eichler--Shimura isomorphism (Remark \ref{rem:eichler-shimura-isom}).
Then the Jacquet--Langlands correspondence (e.g. \cite[Remark 2.2]{prasanna-arithmetic-aspects}) identifies
$\mathcal{M}'_{N^+}$ with $\mathcal{M}^{N^-}_{N}$.
Under the freeness of the (identifed) Hecke module  over $\mathbb{T}(N)^{N^-}_{\mathfrak{m}_f}$ (Assumption \ref{assu:freeness-shimura}), \cite[Lemma 4.17]{ddt} identifies 
\begin{equation}  \label{eqn:congruence-ideals-normalize-shimura-curves}
\xi_{f_B}(N^+, N^-) = \eta_{f}(N^+, N^-) .
\end{equation}
See also \cite[Remark 2.1]{prasanna-arithmetic-aspects}.
We define the \textbf{canonical period for $f_B$} by
$$\Omega^{N^-}_{f} := \dfrac{(4\pi)^k \langle f_B, f_B \rangle_{\Gamma}}{\eta_{f}(N^+, N^-)} $$
where $\langle f_B, f_B \rangle_{\Gamma}$ is the Petersson inner product on $X^{N^-}(N^+)(\mathbb{C})$.


\subsubsection{Prasanna's conjecture}
In \cite[Remark 2.2 and Example 2.3]{prasanna-arithmetic-aspects}, it is expected that
$\mathrm{ord}_\lambda (  \Omega^{N^-}_f / \Omega_f )$ is a $p$-adic unit if $\overline{\rho}$ is irreducible.
Note that it is a different phenomenon from the definite case (cf. Corollary \ref{cor:gross-period-comparison}).
In the indefinite case, it expected that the difference between congruence ideals is encoded in the ratio of Petersson inner products, not the ratio of canonical periods. See \cite[Conjecture 4.2]{prasanna-arithmetic-aspects} for the statement of the conjecture over totally real fields.

Combining all the contents in this section, the following statement immediately follows.
\begin{cor} \label{cor:prasanna-conjecture}
Let $f$ be a newform of weight $k$ and level $\Gamma_0(N)$ with decomposition $N = N^+N^-$ such that
\begin{enumerate}
\item the restriction $\overline{\rho}_f \vert_{G_{\mathbb{Q}(\sqrt{p^*})}}$ is absolutely irreducible where $p^* = (-1)^{\frac{p-1}{2}}p$,
\item $2 \leq k \leq p-1$,
\item $p$, $N^+$, and $N^-$ are pairwisely relatively prime,
\item $N^-$ is square-free, and
\item if a prime $q \equiv \pm 1 \pmod{p}$ and $q$ divides $N^-$, then $\overline{\rho}_f$ is ramified at $q$.
\end{enumerate}
We further assume the following statements:
\begin{itemize}
\item[(a)] $N^-$ is the square-free product of an \emph{even} number of primes (Assumption \ref{assu:indefinite}).
\item[(b)] The Jacquet--Langlands transfer $f_B$ of $f$ is integrally normalized in the sense of \cite{prasanna-annals} (Assumption \ref{assu:integral-shimura}).
\item[(c)] The Hecke module $\mathcal{M}'_{N^+}$ is free of rank two over $\mathbb{T}(N)^{N^-}_{\mathfrak{m}_f}$
(Assumption \ref{assu:freeness-shimura}).
\item[(d)] $\mathrm{ord}_\lambda (  \Omega^{N^-}_f / \Omega_f )$ is a $\lambda$-adic unit.
\end{itemize}
Then
$$\mathrm{ord}_\lambda (  \dfrac{\langle f, f\rangle_{\Gamma_0(N)}}{\langle f_B, f_B \rangle_{\Gamma}} ) = \sum_{q \vert N^-} t_q(f) .$$ 
\end{cor}
\begin{proof}
Since $\mathrm{ord}_\lambda (  \Omega^{N^-}_f / \Omega_f )$ is a $\lambda$-adic unit, 
$\mathrm{ord}_\lambda (  \dfrac{\langle f, f\rangle_{\Gamma_0(N)}}{\langle f_B, f_B \rangle_{\Gamma}} ) $
becomes the ratio between the congruence ideal and the $N^-$-new congruence ideal (with help of (\ref{eqn:congruence-ideals-normalize-shimura-curves})). Then the conclusion follows from Theorem \ref{thm:main_thm}.
\end{proof}

\begin{rem} 
\begin{enumerate}
\item In the case of elliptic curves, Assumption (b) on the integrality can be removed by the geometric method of Ribet--Takahashi. See \cite[$\S$2.2.1]{prasanna-annals}.
\item In the case of weight two forms, Assumption (c) follows from \cite[Corollary 8.11 and Remark 8.12]{helm-israel} under the tame level assumption (5).
In \cite{helm-israel}, although the level is assumed to be square-free, it seems easy to be removed via Ihara's lemma for Shimura curves over $\mathbb{Q}$.
\item 
In the case of elliptic curves, Assumption (d) on the ratio of the canonical periods can be removed by using Faltings' isogeny theorem, but we need it for the higher weight case. See \cite[Example 2.3]{prasanna-arithmetic-aspects}.
\end{enumerate}
\end{rem}
\subsubsection{Remarks on the freeness of the Hecke modules for Shimura curves} \label{subsubsec:remark-freeness-shimura-curves}
We briefly review the development of the freeness of the higher weight Hecke modules for Shimura curves over $\mathbb{Q}$ based on \cite{cheng-multiplicity-one-for-hilbert, cheng-fu_multiplicities}. Let $\Sigma^+$ be the primes in $\Sigma$ not dividing $N^-$ as defined in $\S$\ref{subsubsec:first-approximation}.
\begin{thm}[Cheng, Cheng--Fu]
We assume the following conditions:
\begin{itemize}
\item Assumption \ref{assu:working_assumptions}.(TW).
\item $\overline{\rho}$ occurs in $\mathcal{M}'_{N(\overline{\rho})^+}$.
\item $N(\overline{\rho})$ is square-free.
\item If $\ell$ divides $N^-$ and $\ell^2 \equiv 1 \pmod{p}$, then $\ell$ divides $N(\overline{\rho})$.
\item $\mathrm{End}_{\overline{\mathbb{F}}_p[G_{\mathbb{Q}_p}]} ( \overline{\rho}\vert_{G_{\mathbb{Q}_p}} \otimes \overline{\mathbb{F}}_p) = \overline{\mathbb{F}}_p$ .
\end{itemize}
Then
$\mathcal{M}'_{N^+_{\overline{\rho},\Sigma^+}}$ is free of rank two over $\mathbb{T}(N^+_{\overline{\rho},\Sigma^+} \cdot N^-)^{N^-}_{\mathfrak{m}_{\Sigma^+}}$.
\end{thm}
\begin{proof}
See \cite[Theorem 5.14]{cheng-multiplicity-one-for-hilbert} for the $N^+_{\overline{\rho},\Sigma^+} = N(\overline{\rho})^+$ case, i.e. $\Sigma^+ = \emptyset$.
For general $\Sigma^+$, see \cite[Theorem 3.11]{cheng-fu_multiplicities}, which is based on \cite{diamond_multiplicity_one} and Ihara's lemma for Shimura curves \cite{diamond-taylor-non-optimal}.
\end{proof}
\begin{rem} 
\begin{enumerate}
\item The square-freeness condition of $N(\overline{\rho})$ is imposed in \cite[Page 420]{cheng-multiplicity-one-for-hilbert}.
\item 
The condition $\mathrm{End}_{\overline{\mathbb{F}}_p[G_{\mathbb{Q}_p}]} ( \overline{\rho}\vert_{G_{\mathbb{Q}_p}} \otimes \overline{\mathbb{F}}_p) = \overline{\mathbb{F}}_p$ is imposed in \cite[(3.2)]{cheng-multiplicity-one-for-hilbert}.
It seems that it comes from \cite[(9), Page 30]{taylor-wiles-quaternionic} in order to have the representability of potentially Barsotti--Tate deformation rings (cf.~\cite{cdt}).
We expect that the replacement of this condition by the low weight crystalline condition would not cause any problem.
Indeed, this modification is already spelled out in the global deformation problem in \cite[Definition 5.12]{cheng-multiplicity-one-for-hilbert}.
\end{enumerate}
\end{rem}

%

\bibliographystyle{amsalpha}
\bibliography{kim-ota}

\providecommand{\bysame}{\leavevmode\hbox to3em{\hrulefill}\thinspace}
\providecommand{\MR}{\relax\ifhmode\unskip\space\fi MR }
\providecommand{\MRhref}[2]{%
  \href{http://www.ams.org/mathscinet-getitem?mr=#1}{#2}
}
\providecommand{\href}[2]{#2}
\begin{thebibliography}{KPW17}

\bibitem[B{\"{o}}c07]{bockle-presentations}
Gebhard B{\"{o}}ckle, \emph{Presentations of universal deformation rings},
  {$L$}-functions and {G}alois representations (Cambridge) (David Burns, Kevin
  Buzzard, and Jan Nekov\'{a}\v{r}, eds.), London Math. Soc. Lecture Note Ser.,
  vol. 320, Cambridge University Press, 2007, pp.~24--58.

\bibitem[CDT99]{cdt}
Brian Conrad, Fred Diamond, and Richard Taylor, \emph{Modularity of certain
  potentially {B}arsotti--{T}ate {G}alois representations}, J. Amer. Math. Soc.
  \textbf{12} (1999), no.~2, 521--567.

\bibitem[CF18]{cheng-fu_multiplicities}
Chuangxun Cheng and Ji~Fu, \emph{On multiplicities of {G}alois representations
  in cohomology groups of {S}himura curves}, Math. Res. Lett. \textbf{25}
  (2018), no.~3, 759--782.

\bibitem[CH15]{chida-hsieh-main-conj}
Masataka Chida and Ming-Lun Hsieh, \emph{On the anticyclotomic {I}wasawa main
  conjecture for modular forms}, Compos. Math. \textbf{151} (2015), no.~5,
  863--897.

\bibitem[CH18]{chida-hsieh-p-adic-L-functions}
\bysame, \emph{Special values of anticyclotomic {$L$}-functions for modular
  forms}, J. Reine Angew. Math. \textbf{741} (2018), 87--131.

\bibitem[Che13]{cheng-multiplicity-one-for-hilbert}
Chuangxun Cheng, \emph{Multiplicity one of regular {S}erre weights}, Israel J.
  Math. \textbf{198} (2013), no.~1, 419--466.

\bibitem[CHT08]{clozel-harris-taylor}
Laurent Clozel, Michael Harris, and Richard Taylor, \emph{Automorphy for some
  {$\ell$}-adic lifts of automorphic mod {$\ell$} {G}alois representations},
  Publ. Math. Inst. Hautes \'{E}tudes Sci. \textbf{108} (2008), 1--181.

\bibitem[DDT97]{ddt}
Henri Darmon, Fred Diamond, and Richard Taylor, \emph{Fermat's last theorem},
  Elliptic curves, modular forms {\&} {F}ermat's last theorem ({H}ong {K}ong,
  1993) (Cambridge, {MA}) (John Coates and Shing-Tung Yau, eds.), International
  {P}ress, 1997, Second {E}dition, pp.~2--–140.

\bibitem[DFG04]{diamond-flach-guo}
Fred Diamond, Matthias Flach, and Li~Guo, \emph{The {T}amagawa number
  conjecture of adjoint motives of modular forms}, Ann. {S}ci. {\'{E}}c.
  {N}orm. {S}up\'{e}r. (4) \textbf{37} (2004), no.~5, 663--727.

\bibitem[Dia97a]{diamond-euler}
Fred Diamond, \emph{Congruences between modular forms: raising the level and
  dropping {E}uler factors}, Proc. Natl. Acad. Sci. USA \textbf{94} (1997),
  11143--11146.

\bibitem[Dia97b]{diamond_extension_wiles}
\bysame, \emph{An extension of {W}iles’ results}, Modular {F}orms and
  {F}ermat's {L}ast {T}heorem (Gary Cornell, Joseph Silverman, and Glenn
  Stevens, eds.), Springer, 1997, pp.~475--–489.

\bibitem[Dia97c]{diamond_multiplicity_one}
\bysame, \emph{The {T}aylor--{W}iles construction and multiplicity one},
  Invent. Math. (1997), no.~128, 379--391.

\bibitem[Dim09]{dimitrov-ihara}
Mladen Dimitrov, \emph{On {I}hara's lemma for {H}ilbert modular varieties},
  Compos. Math. \textbf{145} (2009), no.~5, 1114--1146.

\bibitem[DT94a]{diamond-taylor-lifting}
Fred Diamond and Richard Taylor, \emph{Lifting modular mod {$\ell$}
  representations}, Duke Math. J. \textbf{74} (1994), no.~2, 253--269.

\bibitem[DT94b]{diamond-taylor-non-optimal}
\bysame, \emph{Non-optimal levels of mod {$\ell$} modular representations},
  Invent. Math. \textbf{115} (1994), 435--462.

\bibitem[Dum15]{dummigan-higher-congruences}
Neil Dummigan, \emph{Level-lowering for higher congruences of modular forms},
  preprint, October 14th 2015.

\bibitem[EPW06]{epw}
Matthew Emerton, Robert Pollack, and Tom Weston, \emph{Variation of {I}wasawa
  invariants in {H}ida families}, Invent. Math. \textbf{163} (2006), no.~3,
  523--580.

\bibitem[FJ95]{faltings-jordan}
Gerd Faltings and Bruce Jordan, \emph{Crystalline cohomology and
  {$\mathrm{GL}(2, \mathbb{Q})$}}, Israel J. Math. \textbf{90} (1995), 1--66.

\bibitem[FKR21]{fakhruddin-khare-ramakrishna-quantitative}
Najmuddin Fakhruddin, Chandrashekhar Khare, and Ravi Ramakrishna,
  \emph{Quantitative level lowering for {G}alois representations}, J. Lond.
  Math. Soc. \textbf{103} (2021), no.~1, 250--287.

\bibitem[FL82]{fontaine-laffaille}
Jean-Marc Fontaine and Guy Laffaille, \emph{Construction de repr{\'e}sentations
  {$p$}-adiques}, Ann. Sci. Ecole Norm. Sup. (4) \textbf{15} (1982), no.~4,
  547--608.

\bibitem[FPR94]{fontaine-perrin-riou}
Jean-Marc Fontaine and Bernadette Perrin-Riou, \emph{Autour des conjectures de
  {B}loch et {K}ato: cohomologie galoisienne et valeurs de fonctions {$L$}},
  Motives (Providence, RI) (Uwe Jannsen, Steven Kleiman, and Jean-Pierre Serre,
  eds.), vol.~55, Proc. Sympos. Pure Math., no.~1, American Mathematical
  Society, 1994, pp.~599--706.

\bibitem[Gre10]{greenberg-surjectivity}
Ralph Greenberg, \emph{Surjectivity of the global-to-local map defining a
  {S}elmer group}, Kyoto J. Math. \textbf{50} (2010), no.~4, 853--888.

\bibitem[Hel07]{helm-israel}
David Helm, \emph{On maps between modular {J}acobians and {J}acobians of
  {S}himura curves}, Israel J. Math. \textbf{160} (2007), 61--117.

\bibitem[Hid81]{hida-invent-1981}
Haruzo Hida, \emph{Congruence of cusp forms and special values of their zeta
  functions}, Invent. Math. \textbf{63} (1981), no.~2, 225--–261.

\bibitem[Hid06]{hida-hilbert}
\bysame, \emph{Hilbert modular forms and {I}wasawa theory}, Oxford Math.
  Monogr., Oxford {S}cience {P}ublications, 2006.

\bibitem[Hid12]{hida-geometric-modular-forms}
\bysame, \emph{Geometric modular forms and elliptic curves}, second ed., World
  Scientific Publishing, 2012.

\bibitem[Kha03]{khare-R=T}
Chandrashekar Khare, \emph{On isomorphisms between deformation rings and
  {H}ecke rings}, Invent. Math. \textbf{154} (2003), no.~1, 199--222, with
  appendix by {G}ebhard {B}\"{o}ckle.

\bibitem[Kis09]{kisin-2-adic-barsotti-tate}
Mark Kisin, \emph{Modularity of 2-adic {B}arsotti--{T}ate representations},
  Invent. Math. \textbf{178} (2009), 587--–634.

\bibitem[KPW17]{kim-pollack-weston}
Chan-Ho Kim, Robert Pollack, and Tom Weston, \emph{On the freeness of
  anticyclotomic {S}elmer groups of modular forms}, Int. J. Number Theory
  \textbf{13} (2017), no.~6, 1443--1455.

\bibitem[KR03]{khare-ramakrishna}
Chandrashekar Khare and Ravi Ramakrishna, \emph{Finiteness of {S}elmer groups
  and deformation rings}, Invent. Math. \textbf{154} (2003), 179--198.

\bibitem[KW09]{khare-wintenberger-1}
Chandrashekar Khare and Jean-Pierre Wintenberger, \emph{Serre{'}s modularity
  conjecture {(I)}}, Invent. Math. \textbf{178} (2009), no.~3, 485--504.

\bibitem[LJ97]{lenstra-complete-intersections}
Hendrik~Willem Lenstra~Jr., \emph{Complete intersections and {G}orenstein
  rings}, Elliptic curves, modular forms {\&} {F}ermat's last theorem ({H}ong
  {K}ong, 1993) (Cambridge, {MA}) (John Coates and Shing-Tung Yau, eds.),
  International {P}ress, 1997, Second {E}dition, pp.~248--257.

\bibitem[Lun16]{lundell-level-lowering}
Benjamin Lundell, \emph{Quantitative level lowering}, Amer. J. Math.
  \textbf{138} (2016), no.~2, 419--448.

\bibitem[Maz89]{mazur-deforming}
Barry Mazur, \emph{Deforming {G}alois representations}, Galois groups over
  {$\mathbb{Q}$} {(Berkeley, CA, 1987)} (Yasutaka Ihara, Kenneth Ribet, and
  Jean-Pierre Serre, eds.), MSRI Publication, vol.~16, 1989, pp.~385--437.

\bibitem[Maz97]{mazur-intro-deformation-theory}
\bysame, \emph{An introduction to the deformation theory of {G}alois
  representations}, Modular {F}orms and {F}ermat's {L}ast {T}heorem (Gary
  Cornell, Joseph Silverman, and Glenn Stevens, eds.), Springer, 1997,
  pp.~243--–311.

\bibitem[Pra06]{prasanna-annals}
Kartik Prasanna, \emph{Integrality of a ratio of {P}etersson norms and
  level-lowering congruences}, Ann. of Math. \textbf{163} (2006), no.~3,
  901--967.

\bibitem[Pra08]{prasanna-arithmetic-aspects}
\bysame, \emph{Arithmetic aspects of the theta correspondence and periods of
  modular forms}, Eisenstein Series and Applications (Wee~Teck Gan, Stephen
  Kudla, and Yuri Tschinkel, eds.), Progr. Math., vol. 258, Birkh\"{a}user
  {B}oston, 2008, pp.~251--269.

\bibitem[PW11]{pw-mu}
Robert Pollack and Tom Weston, \emph{On anticyclotomic {$\mu$}-invariants of
  modular forms}, Compos. Math. \textbf{147} (2011), 1353--1381.

\bibitem[Rib90]{inv100}
Kenneth Ribet, \emph{On modular representations of
  {$Gal(\overline{\mathbb{Q}}/\mathbb{Q})$} arising from modular forms},
  Invent. Math. \textbf{100} (1990), no.~2, 437--476.

\bibitem[RT97]{ribet-takahashi}
Kenneth Ribet and Shuzo Takahashi, \emph{Parametrization of elliptic curves by
  {S}himura curves and by classical modular curves}, Proc. Natl. Acad. Sci. USA
  \textbf{94} (1997), 11110--11114.

\bibitem[Sai09]{saito-hilbert}
Takeshi Saito, \emph{Hilbert modular forms and {$p$}-adic {H}odge theory},
  Compos. Math. \textbf{145} (2009), 1081--1113.

\bibitem[Sch68]{schlessinger-criterion}
Michael Schlessinger, \emph{Functors of {A}rtin rings}, Trans. Amer. Math. Soc.
  \textbf{130} (1968), no.~2, 208--222.

\bibitem[Tak01]{takahashi-jnt}
Shuzo Takahashi, \emph{Degrees of parametrizations of elliptic curves by
  {S}himura curves}, J. Number Theory \textbf{90} (2001), no.~1, 74--88.

\bibitem[Tay03]{taylor-icosahedral-2}
Richard Taylor, \emph{On icosahedral {A}rtin representations, {II}}, Amer. J.
  Math. \textbf{125} (2003), no.~3, 549--566.

\bibitem[Ter03]{taylor-wiles-quaternionic}
Lea Terracini, \emph{A {T}aylor--{W}iles system for quaternionic {H}ecke
  algebras}, Compos. Math. \textbf{137} (2003), no.~1, 23--47.

\bibitem[Vat99]{vatsal-cong}
Vinayak Vatsal, \emph{Canonical periods and congruence formulae}, Duke Math. J.
  \textbf{98} (1999), no.~2, 397--419.

\bibitem[Wan23]{haining-heegner-cycles}
Haining Wang, \emph{Indivisibility of {H}eegner cycles over {S}himura curves
  and {S}elmer groups}, J. Inst. Math. Jussieu (2023), to appear.

\bibitem[Wil95]{wiles}
Andrew Wiles, \emph{Modular elliptic curves and {F}ermat's last theorem}, Ann.
  of Math. (2) \textbf{141} (1995), 443--551.

\bibitem[Yu18]{yu-higher-khare-ramakrishna}
Yih-Jeng Yu, \emph{A note on modularity lifting theorems in higher weights},
  Taiwanese J. Math. \textbf{22} (2018), no.~2, 275--300.

\bibitem[Zha14]{wei-zhang-mazur-tate}
Wei Zhang, \emph{Selmer groups and the indivisibility of {H}eegner points},
  Camb. J. Math. \textbf{2} (2014), no.~2, 191--253.

\end{thebibliography}

\end{document}